\documentclass[a4paper,10pt]{article}
\usepackage[margin=1.5in]{geometry}

\usepackage{xurl}
\usepackage{amssymb}
\usepackage{amsmath}
\usepackage{bm}
\usepackage{stmaryrd}
\usepackage{array}
\usepackage{multirow}
\usepackage{empheq}
\usepackage{enumitem}
	\setlist{nosep} 
\usepackage{color}
\usepackage{verbatim}

\usepackage{stmaryrd}
\usepackage{adjustbox}
\usepackage{hyperref}
\definecolor{darkgreen}{rgb}{0.0, 0.5, 0}
\usepackage[numbers,sort]{natbib}
\hypersetup{
 colorlinks   = true, 
 urlcolor     = blue, 
 linkcolor    = darkgreen, 
 citecolor   = darkgreen 
}
\usepackage[nameinlink]{cleveref}
\usepackage[belowskip=-5pt]{subcaption}
\usepackage{grffile}

\usepackage{algorithm,algorithmicx}
\usepackage[noend]{algpseudocode}
\newcommand*\Let[2]{\State #1 $\gets$ #2}

\usepackage{amsthm}
\newtheorem{lemma}{Lemma}
\newtheorem{remark}{Remark}
\newtheorem{assumption}{Assumption}

\newtheorem{theorem}{Theorem}
\newtheorem{corollary}{Corollary}
\crefname{assumption}{Assumption}{Assumptions}

\DeclareMathOperator{\argmin}{argmin }
\DeclareMathOperator{\adj}{adj}
\DeclareMathOperator{\cL}{\widehat{\mathcal{L}}}
\DeclareMathOperator{\cLs}{\widehat{\mathcal{L}}^2}
\DeclareMathOperator{\bVert}{\big\Vert}

\newcommand{\TheTitle}{Fast solution of fully implicit Runge-Kutta and discontinuous
	Galerkin in time for numerical PDEs, Part I: the linear setting}
\newcommand{\TheAuthors}{B.S. Southworth, O.A. Krzysik, W. Pazner, and H. De Sterck}

\title{{\TheTitle}\thanks{BSS was supported by Lawrence Livermore National
      Laboratory under contract B639443, and as a Nicholas C. Metropolis Fellow
      under the Laboratory Directed Research and Development program of Los
      Alamos National Laboratory. OAK acknowledges the support of an Australian
      Government Research Training Program (RTP) Scholarship.}}

\author{
	Ben S. Southworth\thanks{Theoretical Division, Los Alamos National Laboratory,
    U.S.A. (\url{southworth@lanl.gov}),
    \url{http://orcid.org/0000-0002-0283-4928}}
\and
    Oliver A. Krzysik\thanks{School of Mathematics, Monash University,
  	Australia (\url{oliver.krzysik@monash.edu}),
  	\url{https://orcid.org/0000-0001-7880-6512}}
\and
  	Will Pazner\thanks{Center for Applied Scientific Computing, Lawrence Livermore National Laboratory,
    U.S.A. (\url{pazner1@llnl.gov})}
\and
    Hans De Sterck\thanks{Department of Applied Mathematics,
  	University of Waterloo,
  	Waterloo, Canada
  	(\url{hdesterck@uwaterloo.ca})}
}

\ifpdf
\hypersetup{
  pdftitle={\TheTitle},
  pdfauthor={\TheAuthors}
}
\fi

\begin{document}
\maketitle
\allowdisplaybreaks

\begin{abstract}
Fully implicit Runge-Kutta (IRK) methods have many desirable properties as time
integration schemes in terms of accuracy and stability, but {high-order IRK methods
are not commonly} used in practice with numerical PDEs due to the difficulty of solving the stage
equations. This paper introduces a theoretical and algorithmic preconditioning
framework for solving the systems of equations that arise from IRK
methods applied to linear numerical PDEs (without algebraic constraints). This
framework also naturally applies to discontinuous Galerkin discretizations in
time. Under quite general assumptions on the spatial discretization that yield
stable time integration, the preconditioned operator is proven to {have condition
number bounded by a small, order-one constant, independent of the spatial mesh
and time-step size,} and with only weak dependence on number of
stages/polynomial order; for example, the preconditioned operator for 10th-order
Gauss IRK has condition number less than two, \emph{independent of the spatial
discretization {and time step.}} The new method can be used with arbitrary existing
preconditioners for backward Euler-type time stepping schemes, and is amenable
to the use of three-term recursion Krylov methods when the underlying spatial
discretization is symmetric. The new method is demonstrated to be effective on
various high-order finite-difference and finite-element discretizations of
linear parabolic and hyperbolic problems, demonstrating fast, scalable solution
of up to 10th order accuracy. {The new method consistently outperforms existing
block preconditioning approaches, and in several cases, the new method can achieve
4th-order accuracy using Gauss integration with roughly half the number of
preconditioner applications and wallclock time as required using standard
diagonally implicit RK methods.}
\end{abstract}

\section{Introduction}\label{sec:intro}

\subsection{Fully implicit Runge-Kutta}\label{sec:intro:irk}

Consider the method-of-lines approach to the numerical solution of linear
partial differential equations (PDEs), where we discretize in space and arrive
at a system of ordinary differential equations (ODEs) in time,

\begin{align*}
M\mathbf{u}'(t) =  \mathcal{L}\mathbf{u} + \hat{\mathbf{f}}(t)
	\quad\text{in }(0,T], \quad \mathbf{u}(0) = \mathbf{u}_0,
\end{align*}

where $M$ is a mass matrix, $\mathcal{L}\in\mathbb{R}^{N\times N}$ a discrete
linear operator, and $\hat{\mathbf{f}}(t)$ a time-dependent forcing
function.\footnote{Note,
PDEs with an algebraic constraint, such as the divergence-free
constraint in the Stokes equations, instead yield a differential algebraic equation (DAE), which
requires separate careful treatment, and is addressed in a companion paper along with
nonlinearities \cite{irk2}.}
Then, consider time propagation using an $s$-stage Runge-Kutta scheme,
characterized by the Butcher tableau

\begin{align*}
	\renewcommand\arraystretch{1.2}
	\begin{array}
	{c|c}
	\mathbf{c}_0 & A_0\\
	\hline
	& \mathbf{b}_0^T
	\end{array},
\end{align*}

with Runge-Kutta matrix $A_0 = \{a_{ij}\} \in\mathbb{R}^{s\times s}$,
weight vector $\mathbf{b}_0^T = (b_1, \ldots, b_s)^T$, and quadrature
nodes $\mathbf{c}_0 = (c_1, \ldots, c_s)$.

Runge-Kutta methods update the solution using a sum over stage vectors,

\begin{align}\label{eq:update}
\mathbf{u}_{n+1} & = \mathbf{u}_n + \delta t \sum_{i=1}^s b_i\mathbf{k}_i, \\
M\mathbf{k}_i & = \mathcal{L}
	\left(\mathbf{u}_n + \delta t\sum_{j=1}^s a_{ij}\mathbf{k}_j\right) +
	\mathbf{f}(t_n+\delta tc_i).\label{eq:stages}
\end{align}

The stage vectors $\{\mathbf{k}_i\}$ can then be expressed as the solution of the block linear
system,

\begin{align}\label{eq:k0}
\left( \begin{bmatrix} M  & & \mathbf{0} \\ & \ddots \\ \mathbf{0} & & M\end{bmatrix}
	- \delta t \begin{bmatrix} a_{11}\mathcal{L} & ... & a_{1s}\mathcal{L} \\
	\vdots & \ddots & \vdots \\ a_{s1}\mathcal{L} & ... & a_{ss} \mathcal{L} \end{bmatrix} \right)
	\begin{bmatrix} \mathbf{k}_1 \\ \vdots \\ \mathbf{k}_s \end{bmatrix}
& = \begin{bmatrix} \mathbf{f}_1 \\ \vdots \\ \mathbf{f}_s \end{bmatrix},
\end{align}

where $\mathbf{f}_i := \hat{\mathbf{f}}(t_n+\delta tc_i) + \mathcal{L}(t_n+\delta tc_i)\mathbf{u}_n$.
{Since $\mathcal{L}$ is the same at all stages (independent of time), \eqref{eq:k0}
is often expressed in the equivalent compact Kronecker product form}

\begin{align}\label{eq:kron1}
(I\otimes M - \delta t A_0\otimes \mathcal{L})\mathbf{k} & = \mathbf{f}.
\end{align}

The difficulty in fully implicit Runge-Kutta methods (which we will denote IRK) lies in
solving the $Ns\times Ns$ block linear system in \eqref{eq:k0}. This paper focuses on the
simulation of numerical PDEs, where $N$ is typically very large
and $\mathcal{L}$ is highly ill-conditioned. In such cases, direct
solution techniques to solve \eqref{eq:k0} are not a viable option, and fast, parallel
preconditioned iterative methods must be used ({with an effective preconditioner
being the key)}. However, higher-order IRK methods are not commonly employed
in practice due to the difficulties of solving \eqref{eq:k0}. Even for relatively simple
parabolic PDEs where $-\mathcal{L}$ is symmetric positive definite (SPD), \eqref{eq:k0}
is a large nonsymmetric matrix with significant block coupling. {It is well-known in
preconditioning fields, including (algebraic) multigrid, sparse approximate inverses,
block preconditioning, etc., that nonsymmetric operators and/or systems with block
structure (particularly more than two blocks/variables) generally introduce significant
difficulties, and many methods do not extend well to nonsymmetric systems and/or block
systems.}

\begin{remark}[Discontinuous Galerkin (DG) in time]\label{rem:dg}
DG-in-time discretizations of systems of linear
ODEs give rise to linear algebraic systems of the form
\begin{equation} \label{eq:dg-in-time}
	\left( \begin{bmatrix}
		\delta_{11} M  & & \delta_{1s} M \\
		& \ddots \\
		\delta_{s1} M & & \delta_{ss} M
	\end{bmatrix}
	- \delta t \begin{bmatrix}
		t_{11}\mathcal{L}_1 & ... & t_{1s}\mathcal{L}_1 \\
		\vdots & \ddots & \vdots \\
		t_{s1}\mathcal{L}_s & ... & t_{ss} \mathcal{L}_s
	\end{bmatrix} \right)
		\begin{bmatrix} \mathbf{u}_1 \\ \vdots \\ \mathbf{u}_s \end{bmatrix}
		= \begin{bmatrix} \mathbf{r}_1 \\ \vdots \\ \mathbf{r}_s \end{bmatrix}.
\end{equation}
The coefficients $T = \{t_{ij}\}$ correspond to a temporal mass matrix, the
coefficients $\boldsymbol{\delta}= \{\delta_{ij}\}$ correspond to a DG weak
derivative with upwind numerical flux, and the unknowns $\mathbf{u}_i$ are the
coefficients of the polynomial expansion of the approximate solution (for
example, see \cite{hn,Akrivis2011,Lasaint1974,Makridakis2006}). Both of the
coefficient matrices $T, \boldsymbol{\delta} \in \mathbb{R}^{s\times s}$ are
invertible. It can be seen that the algebraic form of the DG in time
discretization is closely related to the implicit Runge-Kutta system
\eqref{eq:k0}. In fact, {in the case where $\mathcal{L}_i = \mathcal{L}$ for
all $i$}, the system \eqref{eq:dg-in-time} can be recast in the form of
\eqref{eq:k0} using the invertibility of $\boldsymbol{\delta}$, with $A_0 =
\boldsymbol{\delta}^{-1} T$. In particular, the degree-$p$ DG method using
$(p+1)$-point Radau quadrature, which is exact for polynomials of degree $2p$,
is equivalent to the Radau IIA collocation method \cite{Makridakis2006}, which
is used for many of the numerical results in \Cref{sec:numerics}. Thus, although
the remainder of this paper focuses on fully implicit Runge-Kutta, the
algorithms developed here can also be applied to DG discretizations in time on
fixed slab-based meshes.
\end{remark}

\subsection{Outline}\label{sec:intro:outline}

This paper develops a novel framework for the solution of fully
implicit Runge-Kutta methods and DG discretizations in time for linear numerical PDEs,
with theoretically guaranteed preconditioning, \emph{independent of the spatial
discretization.} The new method requires the preconditioning of $s$ real-valued matrices
of the form $\gamma M - \delta t\mathcal{L}$ for some $\gamma > 0$, analogous to the
matrices that arise in backward Euler integration, and is easily implemented using existing
preconditioners and parallel software libraries.

\Cref{sec:intro:hist} provides background on why IRK methods are desirable over the simpler and more
commonly used diagonally implicit Runge-Kutta (DIRK) methods, and also provides some historical
context for the preconditioners developed in this work. \Cref{sec:intro:stab} then briefly discusses
stable integration from a method-of-lines perspective and introduces two key elements that will be
used throughout the paper.

\Cref{sec:solve} introduces the new theoretical and algorithmic
framework for solving for the IRK update in \eqref{eq:update}. Theory is developed 
in \Cref{sec:solve:prec} that guarantees the {condition number of the
preconditioned operator is bounded by a small, order-one constant} under basic
assumptions on stability from \Cref{sec:intro:stab}. 
The condition number of the preconditioned system is \emph{asymptotically optimal}
in the sense that it is bounded independent of the spatial mesh and time step,
and with only weak dependence on the number of stages or polynomial order.
For example, the preconditioned
operator for 10th-order Gauss IRK has condition number less than two. Moreover,
the conditioning results are \emph{independent of the underlying spatial
discretization}, making the proposed method robust and, to-some-extent,
black-box. To our knowledge, \cite{jiao2020optimal} is the only other work to
prove guaranteed preconditioning or convergence for solving IRK methods applied
to arbitrary linear PDEs.\footnote{Some papers have considered bounds on the spectral
radius of the preconditioned operator, e.g., \cite{chen16}, but for non-SPD spatial
operators, such bounds can be a poor indicator of convergence, { e.g., see}
\cite{Manteuffel:2019}.} In addition, in contrast to other works that have
considered the preconditioning of \eqref{eq:k0}, the proposed algorithm here (i)
is amenable to short-term Krylov recursion (conjugate gradient (CG)/MINRES) if
$\gamma M - \mathcal{L}$ is, and (ii) only operates on the solution, thus not
requiring the storage of each stage vector.

Numerical results are provided in \Cref{sec:numerics}, demonstrating the new
method for a variety of problems and corresponding preconditioners, including
very high-order finite-difference and DG spatial discretizations of
advection-diffusion equations, and matrix-free continuous Galerkin
discretizations of diffusion equations. The method is shown to be fast and scalable
up to 10th-order accuracy in time, effective on fully advective (hyperbolic)
problems, and, for multiple examples, can obtain 4th-order accuracy with Gauss
integration using roughly half as many preconditioning iterations {and wallclock
time} as needed by standard 4th-order SDIRK schemes.

{The methods are implemented with the MFEM \cite{Anderson2020} library
and available at \url{https://github.com/bensworth/IRKIntegration}.}

\section{Background}\label{sec:background}

\subsection{Motivation and previous work}\label{sec:intro:hist}

Diagonally implicit Runge-Kutta methods (DIRK), where $A_0$ is lower triangular,
are commonly used in practice \cite{kennedy16}. For such schemes, the solution of \eqref{eq:k0}
using a block substitution algorithm requires only $s$ linear solves of systems
of the form $ M - \delta ta_{ii}\mathcal{L}$. Unfortunately, DIRK schemes
suffer from order reduction, where the order of accuracy observed in practice on
stiff nonlinear PDEs or DAEs can be limited to $\approx \min\{ p, q+1\}$ or $q$,
respectively, for formal integration order $p$ and stage-order $q$
\cite{hairer96,kennedy16}.

The stage-order of a DIRK method is at most one
(EDIRK methods, with one explicit stage, have a maximum stage order of two) and,
thus, even a 6th-order DIRK method may only yield first- or second-order
accuracy \cite{butcher00}. In contrast, IRK methods may have arbitrarily high
stage order and, thus, formally high-order accuracy on stiff, nonlinear
problems, and even index-2 DAEs \cite{hairer96}. Although the focus of this paper is linear PDEs
without algebraic constraints, we want to highlight that the theory and
framework developed here is fundamental to a companion paper on nonlinear PDEs
and DAEs \cite{irk2}. Furthermore, for less stiff problems, IRK methods can yield accuracy
as high as order $2s$ for an $s$-stage method, compared with a maximum of $s$ or
$s+1$ for SDIRK methods with reasonable stability properties \cite[Section
IV.6]{hairer96},\cite{kennedy16}. Multistep methods can overcome some of the accuracy
constraints of SDIRK methods, but implicit multistep methods cannot be
A-stable and greater than order two, which is limiting when considering
advection-dominated or hyperbolic problems, where the field-of-values often
push up against the imaginary axis. Furthermore, for problems where symplectic
integration is desirable for conservation, neither linear multistep nor
explicit methods can be generally symplectic {(i.e., for non-separable problems)}
\cite{Hairer.2002}. Although DIRK methods
can be symplectic, they are limited to at most 4th order and, moreover, known
methods above second order are impractical due to negative diagonal entries
of $A_0$ (leading to a negative shift rather than positive shift of the spatial
discretization) \cite{kennedy16}. Thus, even moderate order symplectic
integration requires IRK methods.

Many papers have considered the solution of \eqref{eq:kron1}, with Butcher
\cite{butcher76} and Bickart's \cite{bickart77} being some of the earliest works,
which develop ways to transform \eqref{eq:kron1} to a simpler form.\footnote{
In Kronecker form \eqref{eq:kron1}, SIRK methods \cite{norsett1976runge} are also
relatively straightforward to solve using existing preconditioning techniques.
But, although SIRK methods offer some advantages over DIRK methods, they still lack
the favorable stability and accuracy properties of IRK methods \cite{burrage82,orel91}.}
There, and in many of the works that followed, the goal was to minimize the cost of LU
decompositions used to solve \eqref{eq:kron1}, typically in the context of ODEs.
For large-scale simulation of PDEs, particularly on modern computing architectures,
LU decompositions (or other direct factorizations) are typically not feasible. In this
vein, a number of people have considered preconditioning techniques for \eqref{eq:kron1}
or approximations to \eqref{eq:kron1} on the nonlinear iteration level or time
discretization level. Various block preconditioning/approximation techniques
have been studied, primarily for parabolic problems
\cite{houwen97b,Houwen97c,nissen11,mardel07,staff06,hoffmann97,jay00}, and
multigrid methods for IRK and parabolic problems were developed in \cite{vanlent05}.
New ADI-type preconditioners for IRK methods were developed for parabolic problems
in \cite{chen14} with spectral radius shown to be $<1$ under reasonable assumptions,
and the method extended to the viscous wave equation in first-order form in
\cite{chen16}. More recently, block ILU preconditioners were successfully applied
to a transformed version of \eqref{eq:k0} in \cite{pazner17} on more difficult
nonlinear compressible fluids problems.
A handful of works have also studied linear solvers for DG-in-time discretizations,
primarily for parabolic problems, including block preconditioning approaches
\cite{exh,8jp,27n}, and direct space-time multigrid methods
\cite{gander2016analysis}. In fact, some of the principles used in this paper
are similar to those used in \cite{exh} for space-time DG discretizations of
linear parabolic problems, and some of the theory derived therein is generalized
to non-parabolic/non-SPD operators in this paper.

Despite many papers considering the efficient solution of IRK/DG-in-time methods,
very little has been done in the development and analysis of preconditioning techniques
for non-parabolic problems/non-SPD spatial operators, particularly methods that are
amenable to combine with existing fast, parallel preconditioners ({unlike, e.g.,
the block ILU approach in \cite{Pazner2019a}}). {To our knowledge, no methods have
been developed with theoretical guarantees of effectiveness/robustness for a wide range
of problems. Here, we develop a preconditioning framework for linear PDEs that can
be used with arbitrary existing preconditioners/linear solvers, and is guaranteed
to provide effective preconditioning under only minor assumptions on the definiteness
of the spatial operator (see \Cref{ass:fov}). This provides a robust, almost-black-box
method that can be quickly added to existing codes with implicit integration to
support IRK integration for linear PDEs.}

{C++ Code for the IRK preconditioners developed here is built on the MFEM
library \cite{Anderson2020}, and available at
\url{https://github.com/bensworth/IRKIntegration}.}

\begin{remark}[Growing interest in IRK]
It is worth pointing out that while writing this paper, at least three preprints
have been posted online studying the use of IRK methods for numerical PDEs.
Two papers develop new block preconditioning techniques for parabolic PDEs
\cite{jiao2020optimal,rana2020new}, and one focuses on a high-level numerical
implementation of IRK methods with the Firedrake package \cite{farrell2020irksome}.
\end{remark}

\subsection{A preconditioning framework and stability}\label{sec:intro:stab}

Throughout the paper, we use the reformulation used in, for example,
\cite{pazner17}, where we can pull an $A_0\otimes I$ out of the
fully implicit system in \eqref{eq:k0}, yielding the equivalent problem

\begin{align}\label{eq:keq}
\left( A_0^{-1}\otimes M - \delta t I \otimes \mathcal{L}\right)
	(A_0\otimes I) \mathbf{k}
& = \mathbf{f}.
\end{align}

The off-diagonal block coupling in \eqref{eq:keq} now consists of mass matrices
rather than differential operators, which makes the analysis and solution more
tractable.
The algorithms developed here depend on the eigenvalues of $A_0$ and
$A_0^{-1}$, leading to our first assumption.

\begin{assumption} \label{ass:eig}
Assume that all eigenvalues of $A_0$ (and equivalently $A_0^{-1})$ have positive real part.
\end{assumption}

Recall that if an IRK method is A-stable, irreducible, and $A_0$ is invertible
(which includes DIRK, Gauss, Radau IIA, and Lobatto IIIC methods, among others),
then \Cref{ass:eig} holds \cite{hairer96}; that is, \Cref{ass:eig} is
straightforward to satisfy in practice.

Stability must be taken into consideration when applying ODE solvers within a
method-of-lines approach to numerical PDEs. The Dalhquist test problem extends
naturally to this setting, where we are interested in the stability of the
linear operator $\mathcal{L}$, for the ODE(s)
$\mathbf{u}'(t) = \mathcal{L}\mathbf{u}$, with solution $e^{t\mathcal{L}}\mathbf{u}$.
{A necessary condition for stability is that the eigenvalues of $\delta t\mathcal{L}$
lie within the region of stability for the Runge-Kutta scheme of choice} (e.g., see \cite{reddy92}).
Here we are interested in implicit schemes and, because most implicit Runge-Kutta
schemes used in practice are A- or L-stable, an effectively necessary condition for
stability is that the real part of eigenvalues of $\mathcal{L}$ be nonpositive.
For normal matrices, this requirement ends up being a necessary and sufficient
condition for stability.

For non-normal or non-diagonalizable operators, the analysis is more complicated.
One of the best known works on the subject is by Reddy and Trefethen \cite{reddy92},
where necessary and sufficient conditions for stability are derived as the
$\varepsilon$ pseudo-eigenvalues of $\mathcal{L}$ being within
$\mathcal{O}(\varepsilon) + \mathcal{O}(\delta t)$ of the stability region
as $\varepsilon,\delta t\to 0$. Here we relax this assumption to something
that is more tractable to work with by noting that the $\varepsilon$
pseudo-eigenvalues are contained within the field of values to
$\mathcal{O}(\varepsilon)$ \cite[Eq. (17.9)]{trefethen2005spectra},
where the field of values is defined as

\begin{align}\label{eq:fov}
W(\mathcal{L}) := \left\{ \langle \mathcal{L}\mathbf{x},\mathbf{x}\rangle \text{ : }
	\|\mathbf{x}\| = 1 \right\}.
\end{align}

This motivates the following assumption for the analysis done in this paper:

\begin{assumption} \label{ass:fov}
Let $\mathcal{L}$ be the linear spatial operator, and assume that $W(\mathcal{L}) \leq 0$
(that is, $W(\mathcal{L})$ is a subset of the left half plane (including imaginary axis)).
\end{assumption}

{Note that if $\mathcal{L}$ is normal, then \Cref{ass:fov} is equivalent
to the real parts of the eigenvalues of $\mathcal{L}$ being in the closed left-half plane
since $W(\mathcal{L})$ is the convex hull of the eigenvalues.}

It should be noted that the field of values has an additional connection
to stability. From \cite[Theorem 17.1]{trefethen2005spectra}, we have that
$\|e^{t\mathcal{L}}\|\leq 1$ for all $t\geq 0$ if and only if $W(\mathcal{L}) \leq 0$.
This is analogous to the ``strong stability'' discussed by Leveque
\cite[Chapter 9.5]{leveque2007finite}, as opposed to the weaker (but still
sufficient) condition $\|e^{t\mathcal{L}}\|\leq C$ for all $t\geq 0$ and
some constant $C$. In practice, \Cref{ass:fov} often holds when
simulating numerical PDEs, and in \Cref{sec:solve:prec} it is proven that
\Cref{ass:eig,ass:fov} guarantee the methods proposed here yield
{a small, bounded condition number} of the preconditioned operator.
{Specifically, conditioning depends on the eigenvalues of $A_0$,
but the condition number of the preconditioned operator is $<2.5$ for all
Gauss, Radau, and Lobatto schemes tested here, up to 5 stages, and sees
only slow growth in the total number of RK stages.} It should
also be noted that $\mathcal{L}$ need not be nonsingular.

\section{Preconditioning the stage matrix}\label{sec:solve}

For ease of notation, let us scale both sides of \eqref{eq:keq} by a block
diagonal operator, with diagonal blocks $M^{-1}$, and let

\begin{equation}\label{eq:Minv}
\widehat{\mathcal{L}} := \delta t M^{-1}\mathcal{L},
\end{equation}

for $i=1,...,s$. Now let $\alpha_{ij}$ denote the $ij$-element
of $A_0^{-1}$.\footnote{Note, there are
methods with one explicit stage followed by several fully implicit stages \cite{butcher00}.
In such cases, $A_0$ is not invertible, but the explicit stage can
be eliminated from the system (by doing an explicit time step). The remaining operator
can then be reformulated as in \eqref{eq:keq}.}
Then, solving \eqref{eq:keq} can be effectively reduced to inverting the operator

\begin{align}
\mathcal{M}_s \coloneqq A_0^{-1}\otimes I -  I \otimes \widehat{\mathcal{L}}
& = \begin{bmatrix} \alpha_{11}I - \widehat{\mathcal{L}} & \alpha_{12}I & \cdots & \alpha_{1s}I \\
	\alpha_{21}I & \alpha_{22}I - \widehat{\mathcal{L}} & \cdots & \alpha_{2s}I \\
	\vdots & \vdots & \ddots & \vdots \\ \alpha_{s1}I & \cdots & \cdots & \alpha_{ss}I - \widehat{\mathcal{L}} \end{bmatrix}.
	\label{eq:k1}
\end{align}

We proceed by deriving a closed form inverse of \eqref{eq:k1}, demonstrating
how the Runge-Kutta update in \eqref{eq:update} can then be performed directly
(without forming and saving each stage vector), and developing a preconditioning
strategy to apply this update using existing preconditioners.
{Note, in practice we do \emph{not} directly form $\widehat{\mathcal{L}}$,
as $M^{-1}$ is often a dense matrix. Rather, it is a theoretical tool to
simplify notation; in practice mass-matrix inverses are applied via either
preconditioned CG or direct inverse when feasible (e.g., for DG in space).
See \Cref{alg:irk} for a practical description of the final algorithm with
mass matrices.}

\subsection{An inverse and update for commuting operators}\label{sec:solve:inv}

This section introduces a result similar to Bickart's \cite{bickart77},
but using a different framework. We consider $\mathcal{M}_s$
as a matrix over the commutative ring of linear combinations of $\{I, \widehat{\mathcal{L}}\}$,
and the determinant and adjugate referred to in \Cref{lem:inv} are defined
over matrix-valued elements rather than scalars. For the interested reader,
see \cite{brown1993matrices} for details on matrices and the corresponding
linear algebra when matrix elements are defined over a space of commuting
matrices.

{
\begin{lemma}\label{lem:inv}
Let $\alpha_{ij}$ denote the $(i,j)$th entry of $A_0^{-1}$ and define $\mathcal{M}_s$

as in \eqref{eq:k1}. Let $\det(\mathcal{M}_s)$ be the determinant of $\mathcal{M}_s$,
$\adj(\mathcal{M}_s)$ be the adjugate of $\mathcal{M}_s$, and $P_s(x)$ be the
characteristic polynomial of $A_0^{-1}$. Then, $\mathcal{M}_s$
is invertible if and only if $\det(\mathcal{M}_s)$ is not singular, and
\begin{align*}
\mathcal{M}_s^{-1} = \big(I_s\otimes P_s(\widehat{\mathcal{L}})^{-1}\big)\textnormal{adj}(\mathcal{M}_s),
\end{align*}
\end{lemma}
\begin{proof}
Notice in \eqref{eq:k1} that $\mathcal{M}_s$ is a matrix over the commutative ring
of linear combinations of $I$ and $\widehat{\mathcal{L}}$. A classical result in
matrix analysis \cite{brown1993matrices} tells us that

\begin{align*}
\textnormal{adj}(\mathcal{M}_s)\mathcal{M}_s = \mathcal{M}_s\textnormal{adj}(\mathcal{M}_s)
	= (I_s\otimes \det(\mathcal{M}_s))I.
\end{align*}

Moreover, $\mathcal{M}_s$ is invertible if and only if the determinant of $\mathcal{M}_s$
is invertible, in which case $\mathcal{M}_s^{-1} =
(I_s\otimes \det(\mathcal{M}_s)^{-1})\textnormal{adj}(\mathcal{M}_s)$
\cite[Theorem 2.19 \& Corollary 2.21]{brown1993matrices}.
Moreover, notice that $\mathcal{M}_s$ takes the form $A_0^{-1} - \widehat{\mathcal{L}}I$
over the commutative ring defined above. Analogous to a matrix defined over the real or
complex numbers, the determinant of $A_0^{-1} - \widehat{\mathcal{L}}I$ is the
characteristic polynomial of $A_0^{-1}$ evaluated at $\widehat{\mathcal{L}}$,
which completes the proof.
\end{proof}
}

Returning to \eqref{eq:keq}, {let $\hat{\mathbf{f}} = (I_s\otimes M^{-1})\mathbf{f}$,
which applies the mass-matrix inverse in \eqref{eq:Minv} to the right-hand side.} Then,
we can express the solution for the set of all
stage vectors ${\mathbf{k}} = [\mathbf{k}_1; ...; \mathbf{k}_s]$ as

\begin{align*}
\mathbf{k} &:= \left(I_s\otimes \det(\mathcal{M}_s)^{-1}\right)
	(A_0^{-1}\otimes I)\textnormal{adj}(\mathcal{M}_s)(I_s\otimes M^{-1}){\mathbf{f}},
\end{align*}

where $\mathbf{f} = [\mathbf{f}_1; ...; \mathbf{f}_s]$ (note that
$A_0\otimes I$ commutes with $(I_s\otimes \det(\mathcal{M}_s)^{-1})$).
The Runge-Kutta update is then given by

\begin{align}\nonumber
\mathbf{u}_{n+1} & = \mathbf{u}_n + \delta t\sum_{i=1}^s b_i{\mathbf{k}}_i \\
& = \mathbf{u}_n + \delta t\det(\mathcal{M}_s)^{-1}
	(\mathbf{b}_0^TA_0^{-1}\otimes I)\textnormal{adj}(\mathcal{M}_s)(I_s\otimes M^{-1}){\mathbf{f}}.\label{eq:update2}
\end{align}

\begin{remark}[Implementation \& complexity]
The adjugate consists of linear combinations of $I$ and $\widehat{\mathcal{L}}$, and an
analytical form can be derived for an arbitrary $s\times s$ matrix {for small $s$.}

Applying its action requires a set of vector summations
and matrix-vector multiplications. In particular, the diagonal elements of
$\textnormal{adj}(\mathcal{M}_s)$ are monic polynomials in $\widehat{\mathcal{L}}$ of
degree $s-1$
and off-diagonal terms are polynomials in $\widehat{\mathcal{L}}$ of degree $s-2$.

Returning to \eqref{eq:update2}, we consider two cases. First, if a given Runge-Kutta
scheme is stiffly accurate (for example, Radau IIA methods),
then $\mathbf{b}_0^TA_0^{-1} = [0,...,0,1]$. This yields
the nice simplification that computing the update in \eqref{eq:update2} only requires
applying the last row of $\textnormal{adj}(\mathcal{M}_s)$ to $\hat{\mathbf{f}}$ (in a
dot-product sense) and applying $\det(\mathcal{M}_s)^{-1}$ to the result. From
the discussion above regarding the adjugate structure, applying the last row of
$\textnormal{adj}(\mathcal{M}_s)$ requires $(s-2)(s-1) + (s-1) = (s-1)^2$ matrix-vector
multiplications. Because this only happens once, followed by the linear solve(s),
these multiplications are typically of relatively marginal cost.

In the more general case of non-stiffly accurate methods (for example, Gauss
methods), one can obtain an analytical form for $(\mathbf{b}_0^TA_0^{-1}\otimes
I)\textnormal{adj}(\mathcal{M}_s)$. Each element in this block $1\times s$
matrix consists of polynomials in $\widehat{\mathcal{L}}$ of degree $s-1$
(although typically not monic). Compared with stiffly accurate schemes, this now
requires $(s-1)s$ matrix-vector multiplications, which is $s-1$ more than for
stiffly accurate schemes, but still typically of marginal overall computational
cost. {For more information, see \Cref{alg:irk} and the discussion that follows it.}
\end{remark}

\subsection{Preconditioning by conjugate pairs}\label{sec:solve:prec}

Following the discussion and algorithm developed in \Cref{sec:solve:inv}, the key
outstanding point in computing $\mathbf{u}_{n+1}$ using the update \eqref{eq:update2}
is inverting $P_s(\widehat{\mathcal{L}})$, where $P_s(x)$
is the characteristic polynomial of $A_0^{-1}$ (see \Cref{lem:inv}).

In contrast to much of the early work on solving IRK systems, where LU factorizations
were the dominant cost and system sizes relatively small, explicitly forming and inverting
$P_s(\widehat{\mathcal{L}})$ for numerical PDEs is typically not a viable option in high-performance
simulation on modern computing architectures. Instead, by computing the eigenvalues
$\{\lambda_i\}$ of $A_0^{-1}$, we can express $P_s(\widehat{\mathcal{L}})$ in a factored form,

\begin{align}\label{eq:fac}
P_s(\widehat{\mathcal{L}}) = \prod_{i=1}^s (\lambda_i I - \widehat{\mathcal{L}}),
\end{align}

and its inverse can then be computed by successive applications of
$(\lambda_iI - \widehat{\mathcal{L}})^{-1}$,
for $i=1,...,s$. Unfortunately, eigenvalues of $A_0$ and $A_0^{-1}$ are often
complex, and for real-valued matrices this makes the inverse of individual factors
$(\lambda_iI - \widehat{\mathcal{L}})^{-1}$ more difficult and often impractical
with standard preconditioners and existing software. Moving forward, let
$\lambda := \eta + \mathrm{i}\beta$ denote an eigenvalue of $A_0^{-1}$,
for $\eta, \beta \in \mathbb{R}$, with $\beta \geq 0$ and $\eta > 0$ under \Cref{ass:eig}.

Here, we combine conjugate eigenvalues into quadratic polynomials
that we must precondition, which take the form

\begin{align}\label{eq:imag1}
\begin{split}
\mathcal{Q}_\eta :&= ((\eta + \mathrm{i}\beta)I -
	\widehat{\mathcal{L}})((\eta - \mathrm{i}\beta)I - \widehat{\mathcal{L}}) \\
& = (\eta^2+\beta^2)I - 2\eta \widehat{\mathcal{L}} + \widehat{\mathcal{L}}^2
= (\eta I - \widehat{\mathcal{L}})^2 + \beta^2I.
\end{split}
\end{align}

{We then express \eqref{eq:fac} as a product of $\mathcal{Q}_{\eta_j}$ \eqref{eq:imag1},
for $j=1,...,s/2$, $P_s(\widehat{\mathcal{L}}) = \prod_{j=1}^{s/2}
\mathcal{Q}_{\eta_j}$, and solve each successive quadratic operator $\mathcal{Q}_{\eta_j}$
(note, for odd $s$ there will also be a term $(\lambda_i I - \widehat{\mathcal{L}})$
corresponding to the real eigenvalue of the Butcher tableau).}
In practice, we typically do not want to directly form or precondition a quadratic
operator like \eqref{eq:imag1}, due to (i) the overhead cost of large parallel matrix
multiplication, (ii) the fact that many fast parallel methods such as multigrid are
not well-suited for solving a polynomial in $\widehat{\mathcal{L}}$, and (iii) it is
increasingly common that even $\mathcal{L}$ is only available as a
partially-assembled/matrix-free operator. The point of \eqref{eq:imag1}
is that by considering conjugate pairs of eigenvalues, the resulting operator is real-valued.
{To invert \eqref{eq:fac}, we fully resolve the inverse for one conjugate
pair of eigenvalues \eqref{eq:imag1} before moving onto the next; this avoids
potentially compounding condition numbers if we tried to invert the full
polynomial \eqref{eq:fac} all-at-once. Moreover, then we only need to store a
solver for one pair of eigenvalues at a time.}

\subsection{Condition-number optimal conjugate preconditioning}\label{sec:solve:gamma}

{
This section develops a preconditioner for $\mathcal{Q}_\eta$ such that the
condition number of the preconditioned operator is bounded by a small, order-one
constant, \emph{independent of } $\widehat{\mathcal{L}}$. The preconditioner is
optimal over the space of general preconditioners $(\delta I - \cL)^{-1}(\gamma I - \cL)^{-1}$,
for $\delta, \gamma \in (0,\infty)$, in terms of minimizing the maximum
condition number over {all} $\cL$. Furthermore, the condition number of the
preconditioned system is asymptotically
optimal in the sense that it is bounded independent of $\delta t$
and spatial mesh spacing, $h$, and has only weak dependence on the order of time
integration. The analysis derived herein is based on the assumption that a small
bounded condition number corresponds to better preconditioners for nonsymmetric matrices.
}

{Given that \eqref{eq:imag1} is a quadratic polynomial in
$\widehat{\mathcal{L}}$, consider defining a preconditioner as
a \emph{factored} quadratic polynomial in $\widehat{\mathcal{L}}$,
$[(\delta I - \widehat{\mathcal{L}})(\gamma I - \widehat{\mathcal{L}})]^{-1}$,
for $\gamma,\delta > 0$, where we can invert the two factors separately.
The preconditioned operator then takes the form}

\begin{align}\label{eq:P_gen}
\mathcal{P}_{\delta,\gamma} & \coloneqq
	(\delta I - \widehat{\mathcal{L}})^{-1}(\gamma I - \widehat{\mathcal{L}})^{-1}
		\Big[(\eta I - \widehat{\mathcal{L}})^2 + \beta^2 I\Big].
\end{align}

Such an approach was proven effective for symmetric definite spatial matrices in
\cite{exh}, where it is assumed $\gamma = \delta$, and the constant
$\gamma = \gamma_* = \sqrt{\eta^2+\beta^2}$ {is derived to be optimal
in a certain sense.} \Cref{th:cond} in \Cref{appendix} derives tight bounds
on {the maximum condition number of ${\cal P}_{\delta,\gamma}$ \eqref{eq:P_gen}
over all $\widehat{\mathcal{L}}$ that satisfy \Cref{ass:fov}, and further
derives $\delta,\gamma\in(0,\infty)$ that minimize this upper bound.}
\Cref{cor:cond} {below shows that the optimal factored quadratic preconditioner
\eqref{eq:P_gamma} over all $\delta,\gamma\in(0,\infty)$, in terms of minimizing
the maximum $\ell^2$-condition number over all $\widehat{\mathcal{L}}$ that
satisfy \Cref{ass:fov}, is obtained by setting $\delta = \gamma = \gamma_* \coloneqq
\sqrt{\eta^2+\beta^2}$. The preconditioned operator then takes the form}\footnote{
{Note that in \eqref{eq:P_gen}, the preconditioner ($(\delta I - \widehat{\mathcal{L}})^{-1}
(\gamma I - \widehat{\mathcal{L}})^{-1}$) and the operator ($(\eta I -
\widehat{\mathcal{L}})^2 + \beta^2 I$) commute, and so left and right
preconditioning are equivalent.}}

\begin{align}\label{eq:P_gamma}
\mathcal{P}_{\gamma_*} & \coloneqq
	(\gamma_* I - \widehat{\mathcal{L}})^{-2}
		\Big[(\eta I - \widehat{\mathcal{L}})^2 + \beta^2 I\Big].
\end{align}

\begin{corollary}[Condition-number bounds, independent of $\widehat{\mathcal{L}}$]
\label{cor:cond}
{The maximum $\ell^2$ condition number of ${\cal P}_{\delta,\gamma}$ \eqref{eq:P_gen}
over \emph{all} $\cL$ that satisfy \Cref{ass:fov}} is minimized over
$\delta, \gamma \in (0, \infty)$ when $\delta = \gamma = \gamma_*$, with
\begin{align} \label{eq:gamma*}
\gamma_* = \sqrt{\eta^2 + \beta^2}.
\end{align}
That is to say, the preconditioner $(\gamma_* I - \cL)^{-2}$
is optimal over the space of general
preconditioners $(\delta I - \cL)^{-1}(\gamma I - \cL)^{-1}$, for $\delta,
\gamma \in (0,\infty)$, {in terms of minimizing the maximum
condition number over {all} $\cL$}. Furthermore, the condition number of the
preconditioned operator ${\cal P}_{\gamma_*}$ \eqref{eq:P_gamma} is bounded
for all $\widehat{\mathcal{L}}$ via
\begin{align} \label{eq:kappa_gamma*}
\kappa({\cal P}_{\gamma_*}) \leq \sqrt{1 + \frac{\beta^2}{\eta^2}},
\end{align}
and $\exists$ some $\cL$ such that \eqref{eq:kappa_gamma*} is satisfied
with equality.
\end{corollary}
\begin{proof}
See \Cref{appendix}.
\end{proof}

{
\begin{remark}[Three-term recursion]\label{rem:3}
Note that for a given conjugate pair
of eigenvalue, suppose $(\eta I - \widehat{\mathcal{L}})$ is SPD and
$(\gamma_* I - \widehat{\mathcal{L}})^{-1}$ some SPD preconditioner (with
modified constant, which should not affect definiteness). Then, the associated
quadratic operator $\mathcal{Q}_\eta$ \eqref{eq:imag1} and preconditioner
$(\gamma_* I - \widehat{\mathcal{L}})^{-2}$ are also both SPD. It follows that if
CG/MINRES can be applied to backward Euler or SDIRK schemes, it can also be
applied to the quadratic operators arising here.
\end{remark}
}

\begin{remark}[Mass matrices]
Recall in the finite element context where mass matrices are involved, we defined
$\widehat{\mathcal{L}} := \delta t M^{-1}\mathcal{L}$ {as a theoretical tool.
In practice, we do not form $\widehat{\mathcal{L}}$ directly. Factoring the $M^{-1}$
off of the left, the quadratic polynomial \eqref{eq:imag1} for a given conjugate pair
of eigenvalues can be expressed as}

\begin{align}\label{eq:scaleM}
\mathcal{Q}_\eta = M^{-1}(\eta M - \delta t{\mathcal{L}})M^{-1}(\eta M - \delta t{\mathcal{L}}) + \beta^2I
\end{align}

{To invert $\mathcal{Q}_\eta$ iteratively,
it is best to first scale both sides of the linear system by $M$,
so we iterate on $M\mathcal{Q}_\eta$ (see also \Cref{alg:irk}). Each Krylov or fixed-point iteration
requires applying the operator to compute a residual, and applying $M\mathcal{Q}_\eta$
only requires computing $M^{-1}$ once, while applying $\mathcal{Q}_\eta$ requires
computing $M^{-1}$ twice, thus halving the number of times $M^{-1}$ must be applied
each iteration. Moreover, $\mathcal{Q}_\eta$ is not SPD, but if $M$ and $\mathcal{L}$
are Hermitian, $M\mathcal{Q}_\eta$ \emph{is} SPD, as is the preconditioner
$(\gamma_* I - \delta t{\mathcal{L}})^{-1}M(\gamma_* I - \delta t{\mathcal{L}})^{-1}$,
thus allowing the use of CG or MINRES acceleration analogous to \Cref{rem:3}.}
\end{remark}

\Cref{tab:cond} provides condition number bounds from \Cref{cor:cond} and
\eqref{eq:kappa_gamma*} for Gauss, Radau IIA, and Lobatto IIIC Runge-Kutta methods.
{Note that $\gamma_*$ is different for each conjugate eigenvalue pair and each
IRK method. It also should be pointed out that the formally optimal $\gamma$ in
terms of minimizing condition number is not the same for all $\widehat{\mathcal{L}}$;
rather, here we develop a constant $\gamma_*$ that is \emph{robust and effective} for
all $\widehat{\mathcal{L}}$, and does not require additional analysis (analytical
or numerical) as would be necessary to tune $\gamma$ to a specific operator.}

{

\renewcommand{\arraystretch}{1.15}
\begin{table}[!ht]
  \centering
  \begin{tabular}{| c | c | cc | cc | ccc |}  
  \hline
\multirow{2}{*}{Stages} & 2 & \multicolumn{2}{c}{3} & \multicolumn{2}{|c}{4} & \multicolumn{3}{|c|}{5} \\

& {$\lambda_{1,2}^\pm$} & {$\lambda_1$} & {$\lambda_{2,3}^\pm$} & {$\lambda_{1,2}^\pm$} &
	{$\lambda_{3,4}^\pm$} & {$\lambda_1$} & {$\lambda_{2,3}^\pm$} & {$\lambda_{4,5}^\pm$} \\
\hline
Gauss & 1.15 & 1.00 & 1.38 & 1.61 & 1.04 & 1.00 & 1.83 & 1.13 \\
Radau IIA & 1.22 & 1.00 & 1.51 & 1.79 & 1.05 & 1.00 & 2.05 & 1.15 \\
Lobatto IIIC & 1.41 & 1.00 & 1.79 & 2.12 & 1.06 & 1.00 & 2.42 & 1.17 \\\hline
  \end{tabular}
  \caption{Bounds on $\kappa(\mathcal{P}_{\gamma_*})$ from \Cref{cor:cond} and
  \eqref{eq:kappa_gamma*} for Gauss, Radau IIA, and Lobatto IIIC integration,
  with 2--5 stages. Each column within a given set of stages corresponds
  to either a real eigenvalue, $\lambda_1 = \eta$, or a conjugate pair of eigenvalues,
  e.g., $\lambda_{2,3}^\pm = \eta \pm \mathrm{i}\beta$, of
  $A_0^{-1}$.}\label{tab:cond}
\end{table}

}

{\Cref{cor:cond} introduces a modified constant for preconditioning. To compare
with a naive approach of $\delta=\gamma=\eta$ (i.e., preconditioning by ignoring
the $\beta^2$ term in $\mathcal{Q}_\eta$ \eqref{eq:imag1}), one can derive a
worst-case condition number of $1+\beta^2/\eta^2$ for SPD operators and
$\approx (1 + \beta^2/\eta^2)^{3/2}$ for skew symmetric operators, squaring
and cubing the worst-case condition number derived for $\gamma_*$ in
\Cref{cor:cond}, respectively.} Numerical tests indicate using the modified
constant $\gamma_*$ as opposed to $\eta$ is particularly important for
hyperbolic-type problems, {which tend to have dominant imaginary eigenvalues,
even if the spatial discretization is not skew symmetric.} Indeed, one example
in \Cref{sec:numerics:dg} demonstrates an almost $6\times$ reduction in
iteration count achieved by using $\gamma_*$ instead of $\eta$.

\begin{remark}[Inexact preconditioning]\label{sec:inexact-precond}
In practice, fully converging $(\gamma_* I - \widehat{\mathcal{L}})^{-1}$
each iteration as a preconditioner is often not desirable due to the cost
of performing a full linear solve. Here, we propose
applying a Krylov method to $\mathcal{Q}_\eta:=(\eta^2+\beta^2)I - 2\eta\widehat{\mathcal{L}} +
\widehat{\mathcal{L}}^2$ by computing the operator's action (that is, not fully constructing
it), and preconditioning each Krylov iteration with \textit{two} applications of a sparse
parallel preconditioner for $(\gamma_* I - \widehat{\mathcal{L}})$, approximating the action
of $(\gamma_* I - \widehat{\mathcal{L}})^{-2}$.

Analogous to standard block-preconditioning techniques, this approximate inverse
approach is often (but not always) more efficient than computing a full inverse
each iteration. However, \textit{it is important that the underlying preconditioner
provides a good approximation}.
Fortunately, for difficult problems without highly effective
preconditioners, it is straightforward to apply either multiple inner fixed-point iterations
or an inner Krylov iteration (wrapped with a flexible outer Krylov method
\cite{Notay2000,saad1993flexible}) to ensure robust (outer) iterations.
In \Cref{sec:numerics:dg:diff}, a
numerical example is shown where the proposed method diverges using a single inner
fixed-point iteration as a preconditioner for $(\gamma_* I - \widehat{\mathcal{L}})$, but
three (or more) inner fixed-point iterations yields fast, stable convergence.
\end{remark}

\subsection{Algorithm}\label{sec:solve:alg}

{

We conclude this Section by providing a step-by-step description of the new
method in \Cref{alg:irk}, which computes the solution $\mathbf{u}_{n+1}$ at
time $t_{n+1}$ using the update formula in \eqref{eq:update2}. Bullets
provide additional discussion on some nuances of the implementation.

\begin{algorithm}
  \caption{Advance $\mathbf{u}_n$ to $\mathbf{u}_{n+1}$. Assume even $s$, and $A_0^{-1}$ has $s/2$ complex-conjugate eigenvalue pairs $(\eta_i \pm \mathrm{i} \beta_i)_{i = 1}^{s/2}$.
    \label{alg:irk}}
  \begin{algorithmic}[1]
	\State{Evaluate $\mathbf{f} \equiv \mathbf{f}(\mathbf{u}_n,t_n)$}\Comment{RHS of $Ns \times Ns$ system \eqref{eq:k0}}
	\vspace{1.5ex}
	
	\Statex{// Form RHS of linear system: $\mathbf{z} = [(\mathbf{b}_0^T A_0^{-1} \otimes I_N) \adj({\cal M}_s)] [(I_s \otimes M^{-1})\mathbf{f} ]$}
	
	\Let{$\mathbf{z}$}{$0$}
	
	\For{$i = 1 \to s $}
	
	\Let{$\mathbf{z}$}{$\mathbf{z} + R_i(\widehat{\mathcal{L}}) (M^{-1} \mathbf{f}_i)$} \label{alg:adj}
	
	\EndFor

	\vspace{1ex}
	\Statex{// Solve linear system: $P_s(\widehat{{\cal L}}) \mathbf{y} = \prod \limits_{i=1}^{s/2} {\cal Q}_{\eta_i} \mathbf{y} = \mathbf{z}$,\, where ${\cal Q}_{\eta_i} \coloneqq (\eta_i I - \widehat{{\cal L}})^2 + \beta_i^2 I$}
	
    \For{$i = s/2 \to 1 $}\Comment{Solve ${\cal Q}_{\eta_i} \mathbf{y} = \mathbf{z}$}
    
    \Let{$\mathbf{z}$}{$M \mathbf{z}$}\Comment{Scale ${\cal Q}_{\eta_i} \mathbf{y} = \mathbf{z}$ by $M$}

	\Let{$\mathbf{y}$}{Krylov$\big(M{\cal Q}_{\eta_i}$, $\mathbf{z}$,  ${\cal P} M^{-1}{\cal P}\big)$}
	\Comment{Inner preconditioner ${\cal P} \sim (\gamma_* M - \delta t {\cal L})^{-1}$}
	\label{alg:krylov}
	
	\Let{$\mathbf{z}$}{$\mathbf{y}$}
	\Comment{Set RHS for next $i$}
	
    \EndFor

	\vspace{1ex}
    \Statex{// Get IRK solution at new time: $\mathbf{u}_{n+1} = \mathbf{u}_{n} + \delta t P_s(\widehat{\mathcal{L}})^{-1} \mathbf{z}$}
    
    \Let{$\mathbf{u}_{n+1}$}{$\mathbf{u}_n + \delta t \mathbf{y}$}
        \Comment{IRK solution at $t_{n+1}$}
	
  \end{algorithmic}
\end{algorithm}

\begin{itemize}
\item Line \ref{alg:adj} of \Cref{alg:irk}: The RHS vector in the linear system of \eqref{eq:update2} can be expressed as $\mathbf{z} = [(\mathbf{b}_0^T A_0^{-1} \otimes I_N) \adj({\cal M}_s)] [(I_s \otimes M^{-1})\mathbf{f} ] = \sum_{i = 1}^s  R_i(\widehat{\mathcal{L}}) (M^{-1} \mathbf{f}_i)$. Here $R_i$ is a polynomial of degree $s$ that results from taking the inner product of $\mathbf{b}_0^T A_0^{-1}$ with the $i$th column of the matrix $\adj(A_0^{-1} - x I)$, where $x$ is a scalar variable.

The coefficients of $R_i$ are precomputed with high precision (e.g., in Mathematica), and
after forming the vector $M^{-1} \mathbf{f}_i$, the action of $R_i(\widehat{\mathcal{L}})$ is applied using a Horner-like scheme, which only requires computing the action of $\widehat{\mathcal{L}}$ $s$ times. Recall that the (potentially dense) matrix $\widehat{\mathcal{L}} = \delta t M^{-1} \mathcal{L}$ is not formed, but its action is computed using that of $M^{-1}$ and $\mathcal{L}$.

\item Line \ref{alg:krylov} of \Cref{alg:irk}: $\mathbf{x} \gets \textrm{Krylov}(A,\mathbf{b}, B)$ means apply a Krylov method to solve $A \mathbf{x} = \mathbf{b}$, with left or right preconditioning $B \approx A^{-1}$. In Line \ref{alg:krylov}, the inner preconditioner ${\cal P}$ is some approximation to $(\gamma_* M - \delta t {\cal L})^{-1}$, such as one multigrid iteration (for example, see \Cref{sec:inexact-precond}), and the full preconditioner consists of applying ${\cal P}M^{-1}{\cal P}$. In each Krylov iteration, the operator $M{\cal Q}_{\eta_i}$ is not formed, but its action is computed using a Horner-like scheme.
\end{itemize}

}

\section{Numerical results}\label{sec:numerics}

{Numerical results consider Gauss, RadauIIA, and LobattoIIIC IRK methods, as well
as several SDIRK methods for comparison: 2-stage, 2nd-order L-stable SDIRK \cite[Eq. 221]{kennedy16}
($\gamma = (2-\sqrt{2})/2$), 2-stage, 3rd-order A-stable SDIRK \cite[Eq. 223]{kennedy16}
($\gamma = (3+\sqrt{3})/3$), 3-stage, 3rd-order L-stable SDIRK \cite[Eq. 229]{kennedy16},
3-stage, 4th-order A-stable SDIRK \cite[Eq. (6.18)]{hairer96}, and
5-stage, 4th-order L-stable SDIRK \cite[Table 6.5]{hairer96}.}

{
For some problems, runtime comparisons are made between the current IRK algorithm and those proposed in \cite{staff06} and \cite{rana2020new}. The algorithms from \cite{staff06} and \cite{rana2020new} solve the stage equations \eqref{eq:k0} with an iterative solver, such as GMRES, for example, using a block preconditioner based on the Butcher matrix $A_0$. 

The block preconditioners are chosen to have a block triangular structure, such that they can be applied via forward/backward substitution. During the application of the preconditioners, exact inverses of the diagonal blocks are approximated with an inexpensive iterative method, such as a single multigrid cycle, for example. 

Of all the preconditioners proposed in Staff et al. \cite{staff06}, we show results for the one that uses a lower-triangular splitting of $A_0$, which we refer to as ``GSL'', since we find it has the smallest runtime. 

From Rana et al. \cite{rana2020new} we compare with the $LD$ preconditioner,
which uses a preconditioner based on the $LD$ component of an $LDU$ factorization of $A_0$,
which we refer to as ``LD.''
}

\subsection{Finite-difference advection-diffusion}\label{sec:numerics:fd}

In this section, we consider a constant-coefficient advection-diffusion problem
discretized in space with high-order finite-differences. An exact solution to
this problem is used to demonstrate the high-order accuracy of the IRK methods,
and the robustness of the algorithms developed in the previous section with
respect to mesh resolution. Specifically, we solve the PDE

\begin{align}
\label{eq:FD_ex} u_t + 0.85 u_x + u_y = 0.3 u_{xx} + 0.25 u_{yy} + s(x,y,t),
\quad (x,y,t) \in (-1,1)^2 \times (0,2],
\end{align}

on a periodic spatial domain. The source term $s(x,y,t)$ is chosen such that
the solution of the PDE is
$u(x,y,t)=\sin^4(\pi/2[x-1-0.85t]) \sin^4(\pi/2 [y-1-t]) \exp(-[0.3+0.25]t)$.

We consider tests using IRK methods of orders three, four, seven, and eight. The
3rd- and 4th-order IRK methods are paired with 4th-order
central-finite-differences in space, and the 7th- and 8th-order methods with
8th-order central-finite-differences in space. In all cases, a time-step of
$\delta t = 2 h$ is used, with $h$ denoting the spatial mesh size, and results
are run on four cores. Due to the
diffusive, but non-SPD nature of the spatial discretization, we apply GMRES(30)
preconditioned by a classical algebraic multigrid (AMG) method in the
\textit{hypre} library \cite{Falgout:2002vu}. Specifically, we
use classical interpolation (type 0), Falgout coarsening (type 6) with a strength
tolerance $\theta_C = 0.25$, zero levels of aggressive coarsening, and
$L_1$-Gauss--Seidel relaxation (type 8), with a relative stopping
tolerance of $10^{-13}$. A single iteration of AMG is applied to approximate
$(\gamma_* I - \delta t {\cal L})^{-1}$.

In \Cref{fig:FD_ex}, discretization errors are shown for different IRK
methods, alongside the average number of AMG iterations needed per time step.
The expected asymptotic convergence rates (black dashed lines in the left panel)
are observed for all discretizations.\footnote{An exception here is A--SDIRK(4),
which appears to be converging with a rate closer to three than four; however,
further decreasing $\delta t$ (not shown here) confirms 4th-order convergence
is achieved eventually.}

\begin{figure}[!htb]
\centerline{
\includegraphics[scale = 0.41]{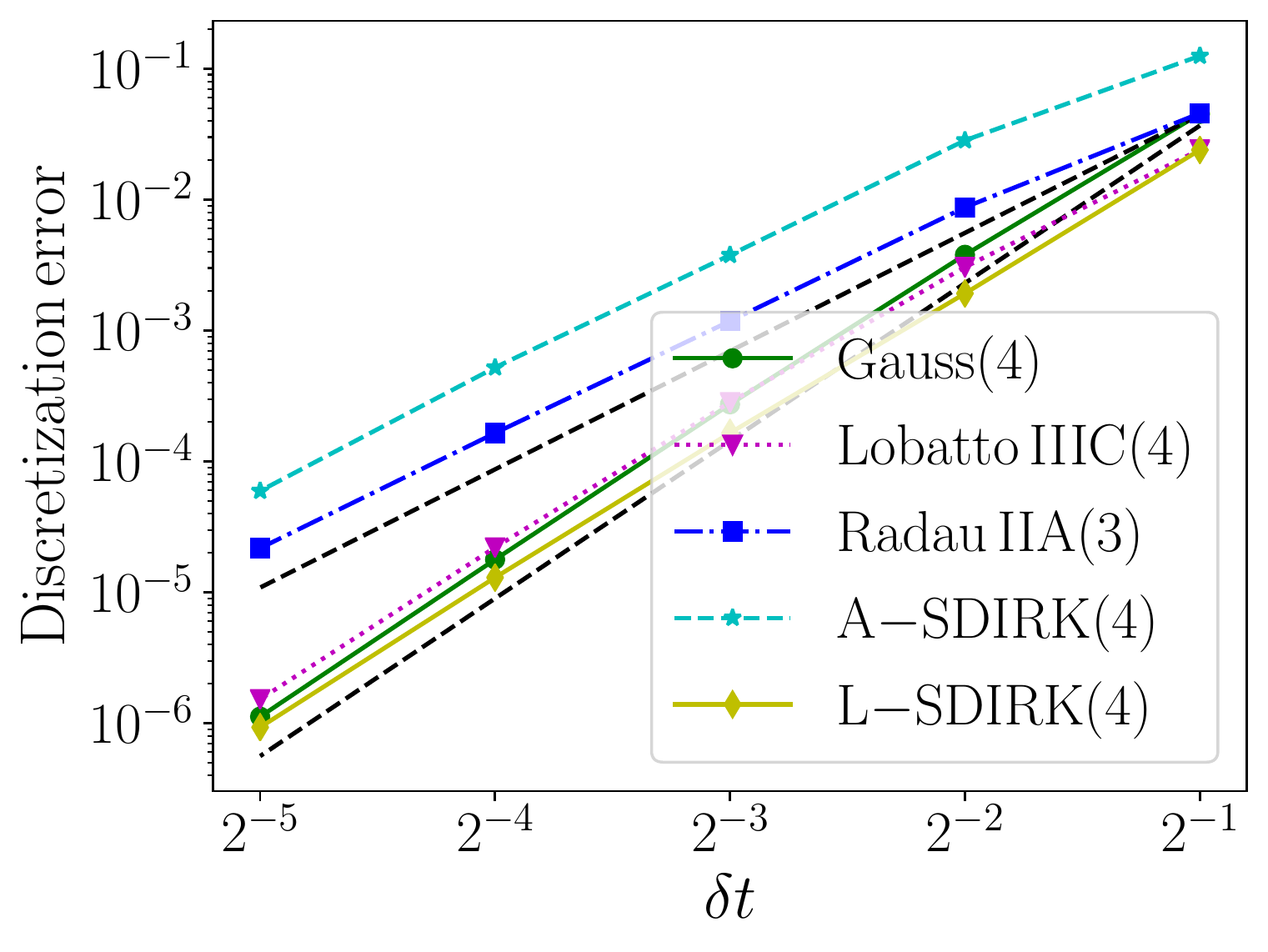}
\quad
\includegraphics[scale = 0.41]{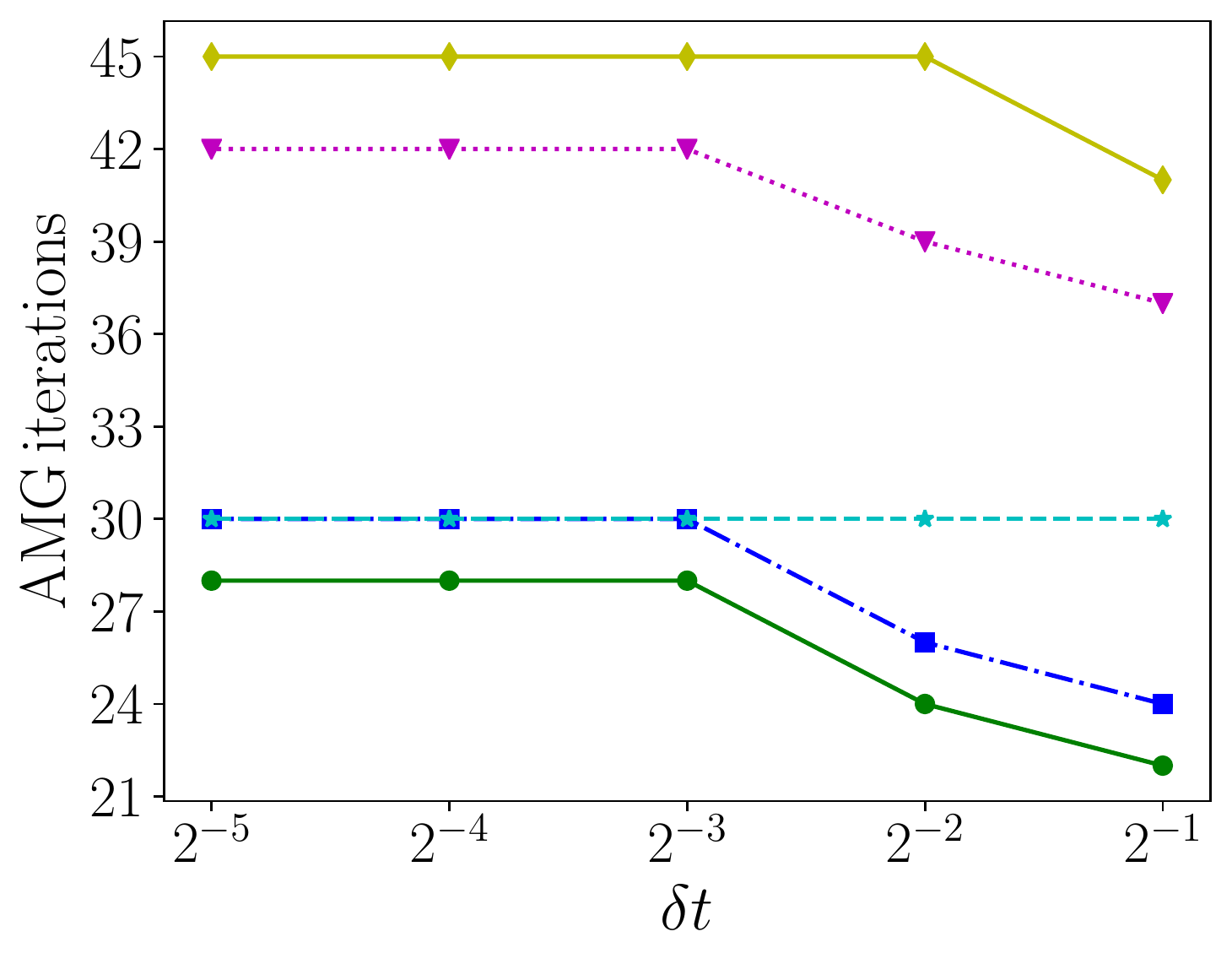}}
\centerline{
\includegraphics[scale = 0.41]{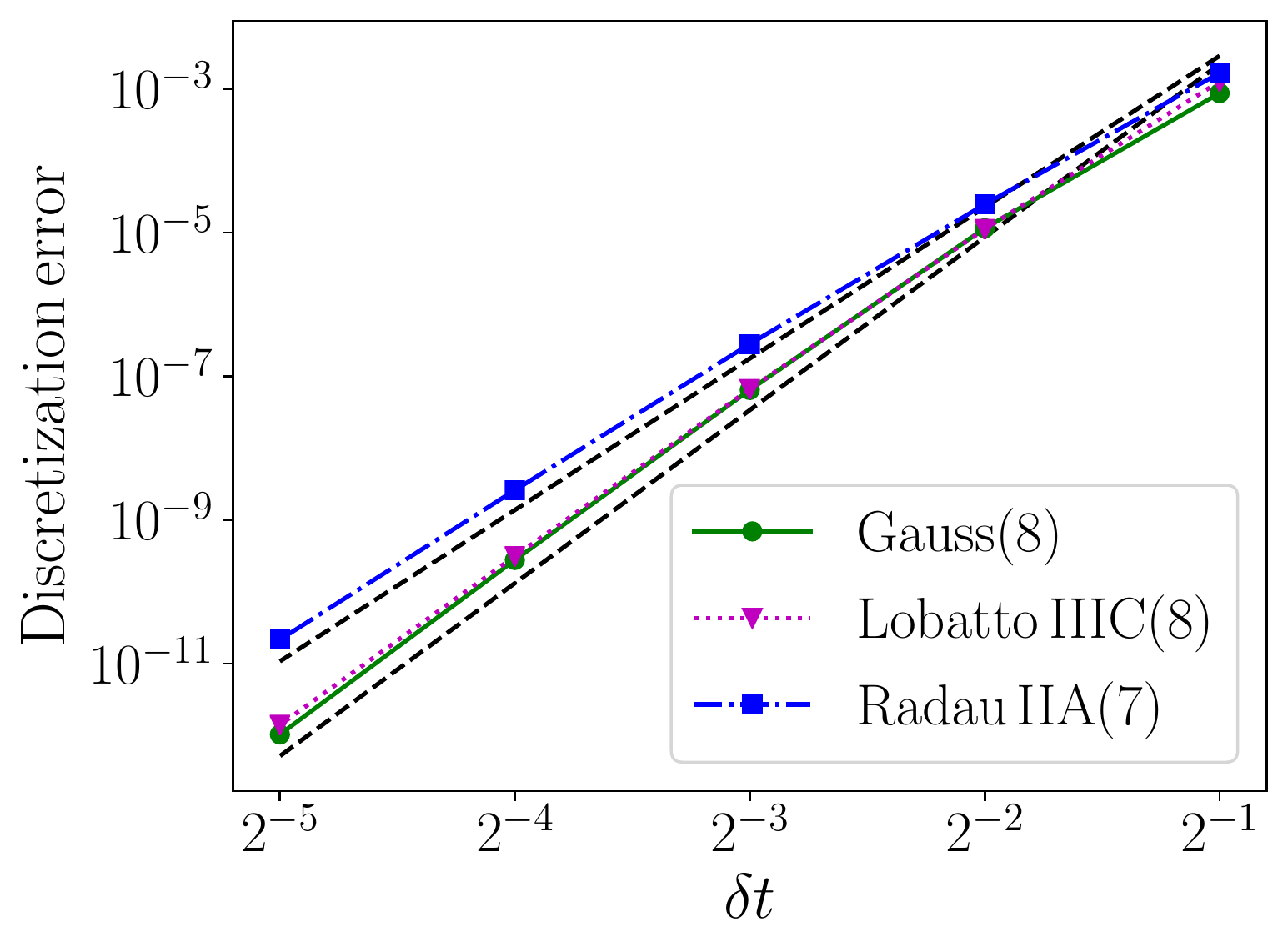}
\quad
\includegraphics[scale = 0.41]{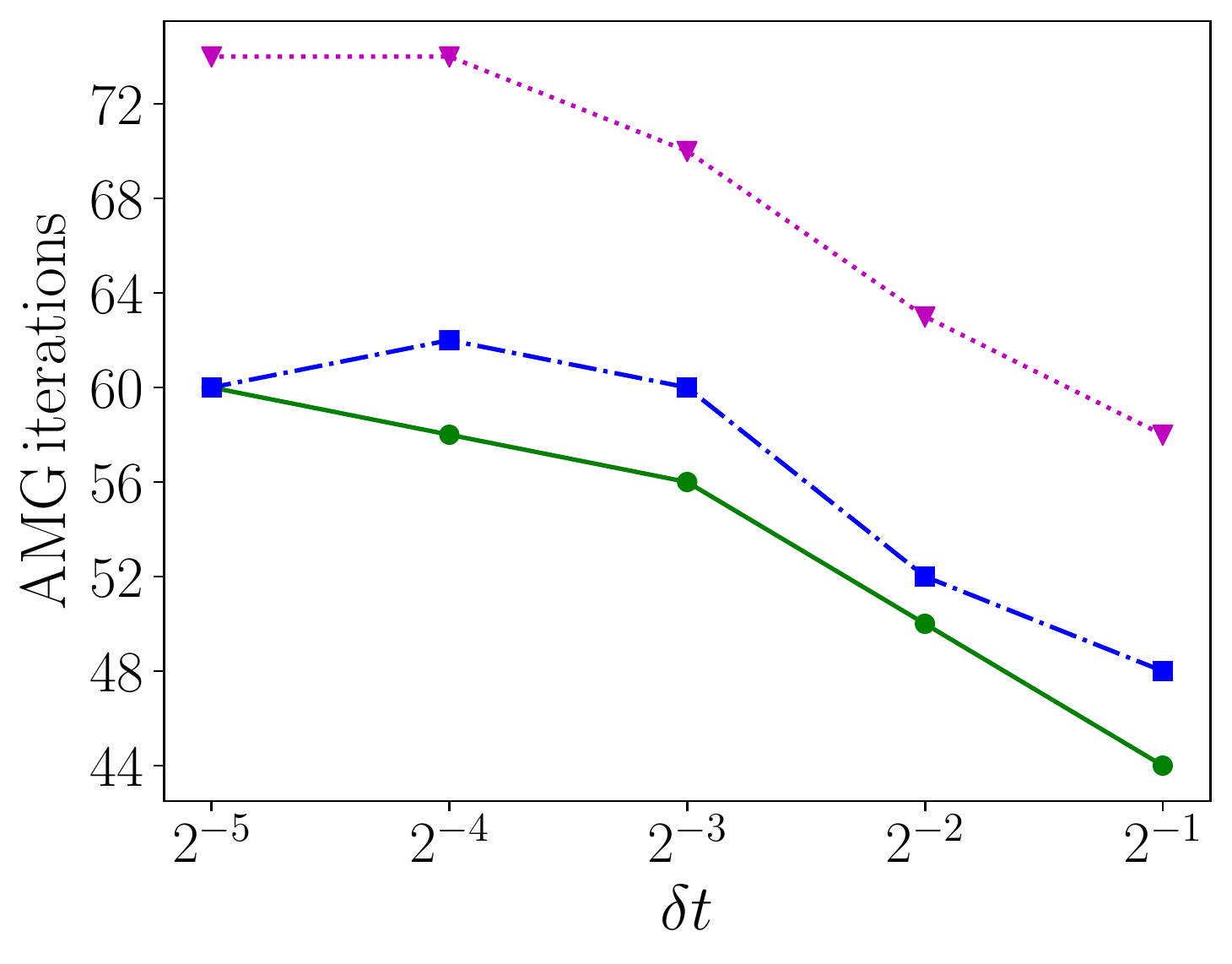}}
\caption{Finite-difference advection-diffusion problem \eqref{eq:FD_ex}. $L_{\infty}$-discretization errors at $t = 2$ as a function of time-step $\delta t$ are shown on
the left for various discretizations of approximately 4th order (top) and 8th order (bottom). Black, dashed lines with slopes of three and four are shown (top), as are those with slopes of seven and eight (bottom). Plots on the right show the average number of AMG iterations per time step.
For time-step size $\delta t = 2^{-\ell}$, the linear systems are of size $n_x \times n_y = 2^{\ell + 2} \times 2^{\ell + 2}$. }
\label{fig:FD_ex}
\end{figure}

The preconditioner appears robust with respect to mesh and problem size, since
the average number of AMG iterations per time step (which is a proxy for the
number of GMRES iterations) remains roughly constant as the the mesh is refined.
Of the fully implicit methods, the Gauss methods require the fewest AMG iterations, closely
followed by Radau IIA methods, with the Lobatto IIIC methods requiring the most
AMG iterations. This is
consistent with the theoretical estimates in \Cref{tab:cond}.
Note that while Gauss and Radau IIA methods have very similar
iteration counts, Gauss converges at one order faster, which can be seen in the
left-hand panel of the figure.

Considering the lower-order methods in the top
row of \Cref{fig:FD_ex}, L--SDIRK(4) (see \cite[Table 6.5]{hairer96}), a 5-stage, 4th-order, L-stable SDIRK
method requires the most AMG iterations of all methods. A--SDIRK(4) (see \cite[eq. (6.18)]{hairer96}), a 3-stage,
4th-order, A-stable SDIRK method, requires far fewer AMG iterations than
L--SDIRK(4). However, A--SDIRK(4) yields a significantly larger discretization
error than the other 4th-order schemes, and takes longer to reach its asymptotic
convergence rate. Thus, in terms of solution accuracy as a function of computational
work, the new preconditioner with 4th-order Gauss integration is the clear winner
for this particular test problem, requiring roughly half the AMG iterations
of the commonly used L-stable SDIRK4 scheme.

\Cref{tab:FDex_block_comparison} shows the runtime of the block-preconditioning approaches of GSL \cite{staff06} and LD \cite{rana2020new} (see the introduction of \Cref{sec:numerics}) relative to the current approach that uses complex-conjugate preconditioning. 

For the GSL and $LD$ solves, GMRES(30) is used as the solver with a relative stopping tolerance of $10^{-13}$, and a single iteration of AMG is used to approximate the inverses of the diagonal blocks in the block lower triangular preconditioners. These AMG methods use the same settings as those described above.

In all cases, the runtime of the current IRK algorithm is smaller than those using block preconditioning, demonstrating the competitiveness of our approach with existing ones, and the advantages of using an optimized preconditioner. 

\begin{table}[!ht]
\renewcommand{\arraystretch}{1.15}
  \centering
  \begin{tabular}{| c | c | c | c |}  
  \hline
  & \multicolumn{3}{c|}{Runtime relative to current approach} \\\hline
& Current & GSL \cite{staff06} & $LD$ \cite{rana2020new} \\\hline
  Gauss(4) & 1 & 1.24 & 1.21 \\\hline
 Radau IIA(3) & 1 & 1.44 & 1.16 \\\hline
  Lobatto IIIC(4) & 1 & 2.09 & 1.86 \\\hline
  Gauss(8) & 1 & 1.64 & 1.57 \\\hline
  Radau IIA(7) & 1 & 1.92 & 1.64 \\\hline
  Lobatto IIIC(8) & 1 & 2.97 & 2.22 \\\hline
  \end{tabular}
  \caption{Finite-difference advection-diffusion problem \eqref{eq:FD_ex}. Runtime of block-preconditioning IRK algorithms GSL \cite{staff06} and $LD$ \cite{rana2020new} relative to the runtime of the current algorithm that uses complex-conjugate preconditioning.Each relative runtime measurement has been calculated as the mean of the relative runtimes for solving each of the problem sizes shown in \Cref{fig:FD_ex}.
  }
  \label{tab:FDex_block_comparison}
\end{table}

\subsection{DG in space advection-diffusion}\label{sec:numerics:dg}

Here we consider a more difficult advection-diffusion problem, discretized using
high-order DG finite elements {in space (independent of the time
discretization; i.e., no relation to \Cref{rem:dg}). We demonstrate the
effectiveness of the new preconditioning and ``optimal $\gamma_*$'' on more complex
flows (\Cref{sec:numerics:dg:const}), examine order reduction in DIRK and IRK methods
(\Cref{sec:numerics:dg:red}),
study the use of multiple ``inner'' preconditioning iterations to approximate
$(\gamma_* M - \delta t \mathcal{L})^{-1}$ (or even inner Krylov acceleration;
\Cref{sec:numerics:dg:diff}),
and finally make a comparison with other state-of-the-art IRK solvers from
\cite{staff06, rana2020new,pazner17} (\Cref{sec:numerics:dg:comp}).}

The governing equations in spatial domain $\Omega = [0,1] \times [0,1]$ are given by
\begin{equation} \label{eq:adv-diff}
	u_t + \nabla \cdot ( \bm\beta u  - \varepsilon \nabla u ) = f
\end{equation}
where $\bm\beta(x,y) : = (\cos(4\pi y), \sin(2 \pi x))^T$
is the prescribed velocity field and $\varepsilon$ the diffusion coefficient.
Dirichlet boundary conditions are weakly enforced on $\partial\Omega$, and
\eqref{eq:adv-diff} is discretized with an upwind DG method \cite{Cockburn2001},
where diffusion terms are treated with the symmetric interior penalty method
\cite{Arnold1982,Arnold2002}. The resulting finite element problem is to find
$u_h \in V_h$ such that, for all $v_h \in V_h$,
\[
	\begin{multlined}
	\int_\Omega \partial_t (u_h) v_h \, dx
	- \int_\Omega u_h \bm\beta \cdot \nabla_h v_h \, dx
	+ \int_\Gamma \widehat{u_h} \bm\beta \cdot \llbracket v_h \rrbracket \, ds
	+ \int_\Omega \nabla_h u_h \cdot \nabla_h v_h \, d x \\
	- \int_\Gamma \{ \nabla_h u_h \} \cdot \llbracket v_h \rrbracket \, ds
	- \int_\Gamma \{ \nabla_h v_h \} \cdot \llbracket u_h \rrbracket \, ds
	+ \int_\Gamma \sigma \llbracket u_h \rrbracket \cdot \llbracket v_h \rrbracket \, ds
	= \int_\Omega f v_h \, dx,
	\end{multlined}
\]
where $V_h$ is the DG finite element space consisting of piecewise polynomials
of degree $p$ defined on elements of the computational mesh $\mathcal{T}$ of the
spatial domain $\Omega$. No continuity is enforced between mesh elements. Here,
$\nabla_h$ is the broken gradient, $\Gamma$ denotes the skeleton of the mesh,
and $\{ \cdot \}$ and $\llbracket \cdot \rrbracket$ denote the average and jump
of a function across a mesh interface. $\widehat{u_h}$ is used to denote the
upwind numerical flux. {The parameter $\sigma$ is the \textit{interior
penalty parameter}, which must be chosen sufficiently large to obtain a stable
discretization \cite{Arnold2002}. In particular, we choose $\sigma \sim p^2/h$;
see also \cite{Shahbazi2005} for an explicit expression for this parameter.} This
discretization has been implemented in the MFEM finite element framework
\cite{Anderson2020}, and uses AMG preconditioning with approximate ideal
restriction (AIR) \cite{Manteuffel:2019,Manteuffel:2018} (after first scaling by
the inverse of the block-diagonal mass matrix).

\subsubsection{Order reduction and wall-clock-time}
\label{sec:numerics:dg:red}

{
We begin by considering a manufactured solution. We set diffusion coefficient
$\epsilon = 10^{-4}$ and choose initial conditions,
Dirichlet boundary conditions, and time-dependent forcing function such
that \eqref{eq:adv-diff} with constant coefficient advection, $\boldsymbol{\beta} = [1,1]$
satisfies the exact space-time solution $u_* = \sin(2\pi(x-t))\sin(2\pi(y-t))$. This
results in weakly imposing time-dependent Dirichlet boundary conditions, a constraint
known to cause order reduction in DIRK methods due to low stage order
\cite{rosales2017spatial}. 4th-order elements are used on a mesh with $h\approx
0.0078$, which results in finite element approximation of the initial condition
with $\ell^2$-error $3\cdot 10^{-12})$ (i.e., this is roughly the expected limit
of accuracy that can be obtained after time integration).

We first consider the order of convergence for L-stable SDIRK methods compared
with RadauIIA and Gauss IRK methods. Error is measured at time $t=2$ against the exact
solution in the $\ell^2$-norm and a broken $\mathcal{H}^1$-norm,
and results are shown in \Cref{fig:dg-err}. Order reduction is most pronounced for
4th-order L-stable SDIRK, achieving only second order in the broken $\mathcal{H}^1$-norm
and order $2.7$ in the $\ell^2$-norm as $\delta t\to 0$. 3rd-order L-stable SDIRK
achieves the expected order in $\ell^2$-norm, but only 2nd-order in broken
$\mathcal{H}^1$-norm. In contrast, 3rd order Radau achieves just under 3rd order
in the $\mathcal{H}^1$-norm and 3rd order in $\ell^2$, and 4th-order Gauss achieves
4th order in both norms for sufficiently small $\delta t$. Most higher-order Radau
and Gauss methods observe some order reduction in both norms, achieving
accuracy somewhere between the stage-order \cite[Section IV.5]{hairer96} ($s$)
and full order ($2s-1$ and $2s$, respectively), but still obtain significantly
smaller error than the lower-order schemes.

\begin{figure}[!htb]
  \centering
  \begin{subfigure}[b]{0.475\textwidth}
    \includegraphics[width=\textwidth]{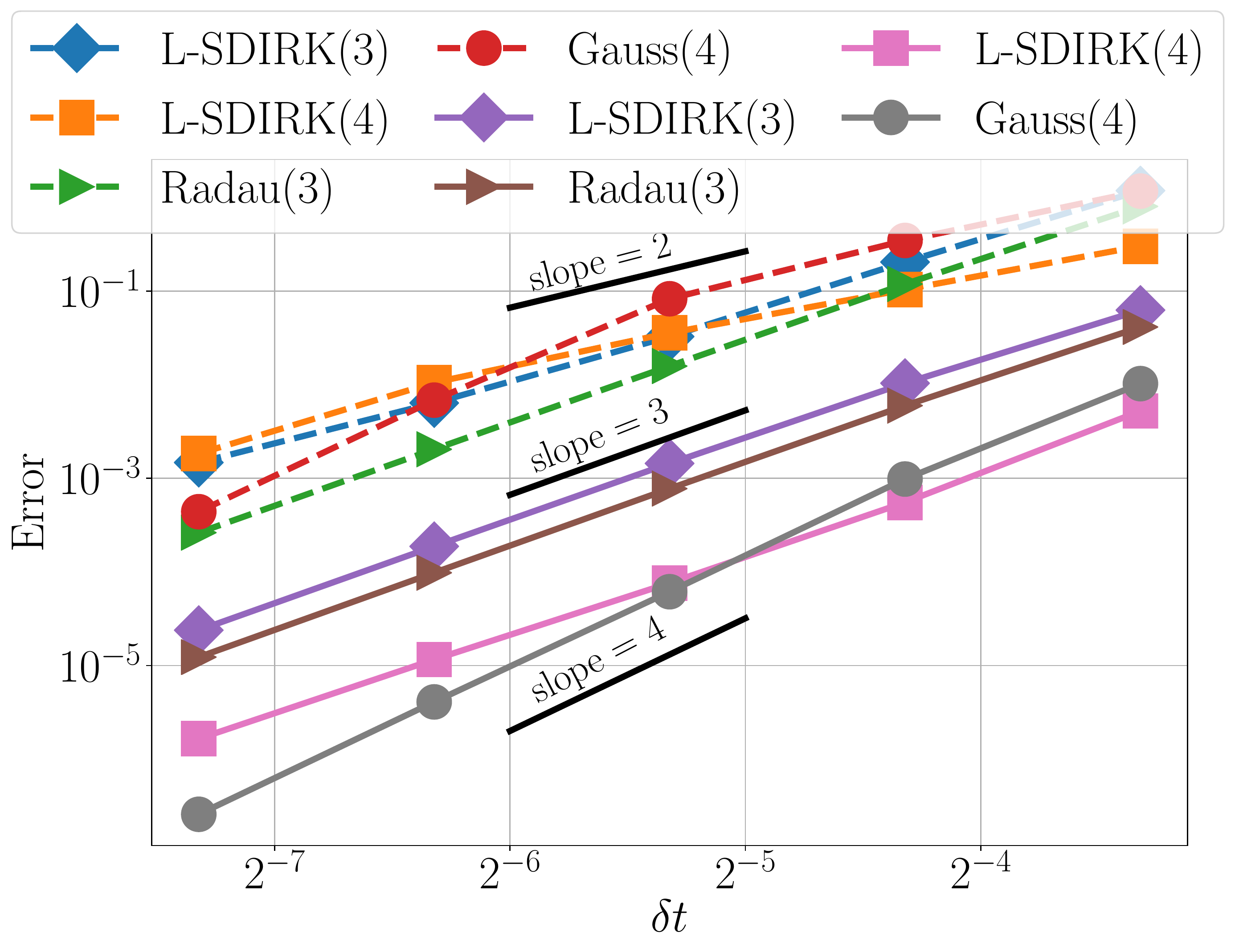}
    
  \end{subfigure}
   \begin{subfigure}[b]{0.475\textwidth}
    \includegraphics[width=\textwidth]{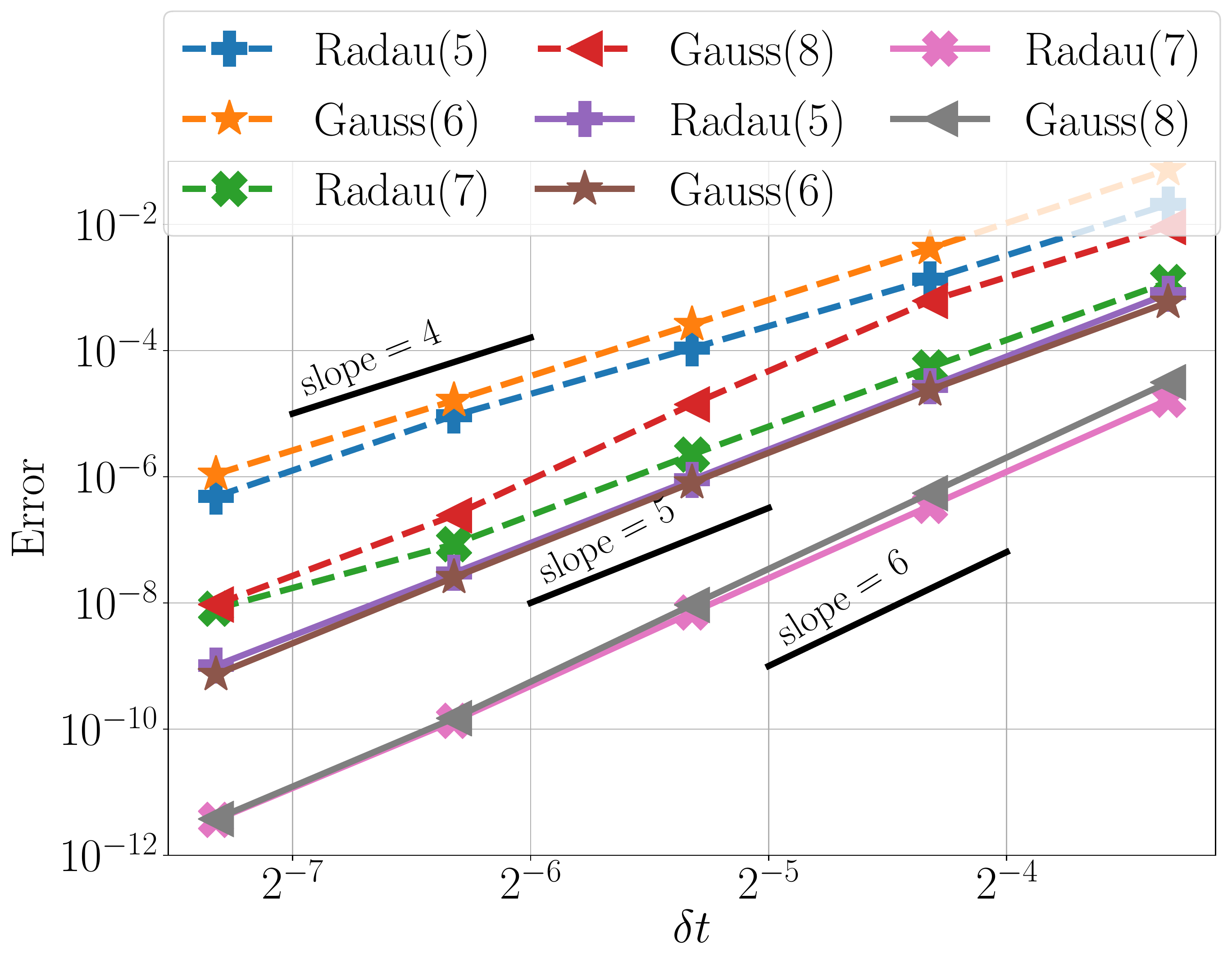}
    
  \end{subfigure}
  \\\vspace{2ex}
      \caption{$\ell^2$-error (solid lines) and broken $\mathcal{H}^1$-error (dashed lines) at time $t=2$.}
      \label{fig:dg-err}
\end{figure}

We now compare accuracy as a function of wallclock time of standard A-stable and
L-stable SDIRK methods with IRK methods using the preconditioners developed here.
\emph{For a given wall-clock time, the IRK methods of a given order yield smaller
error than equivalent-order SDIRK methods in all cases considered.}
Moreover, for a fixed wallclock time, the high-order IRK methods can achieve as
much as two orders of magnitude reduction in error compared with lower-order schemes.

\begin{figure}[!htb]
  \centering
  \begin{subfigure}[b]{0.475\textwidth}
    \includegraphics[width=\textwidth]{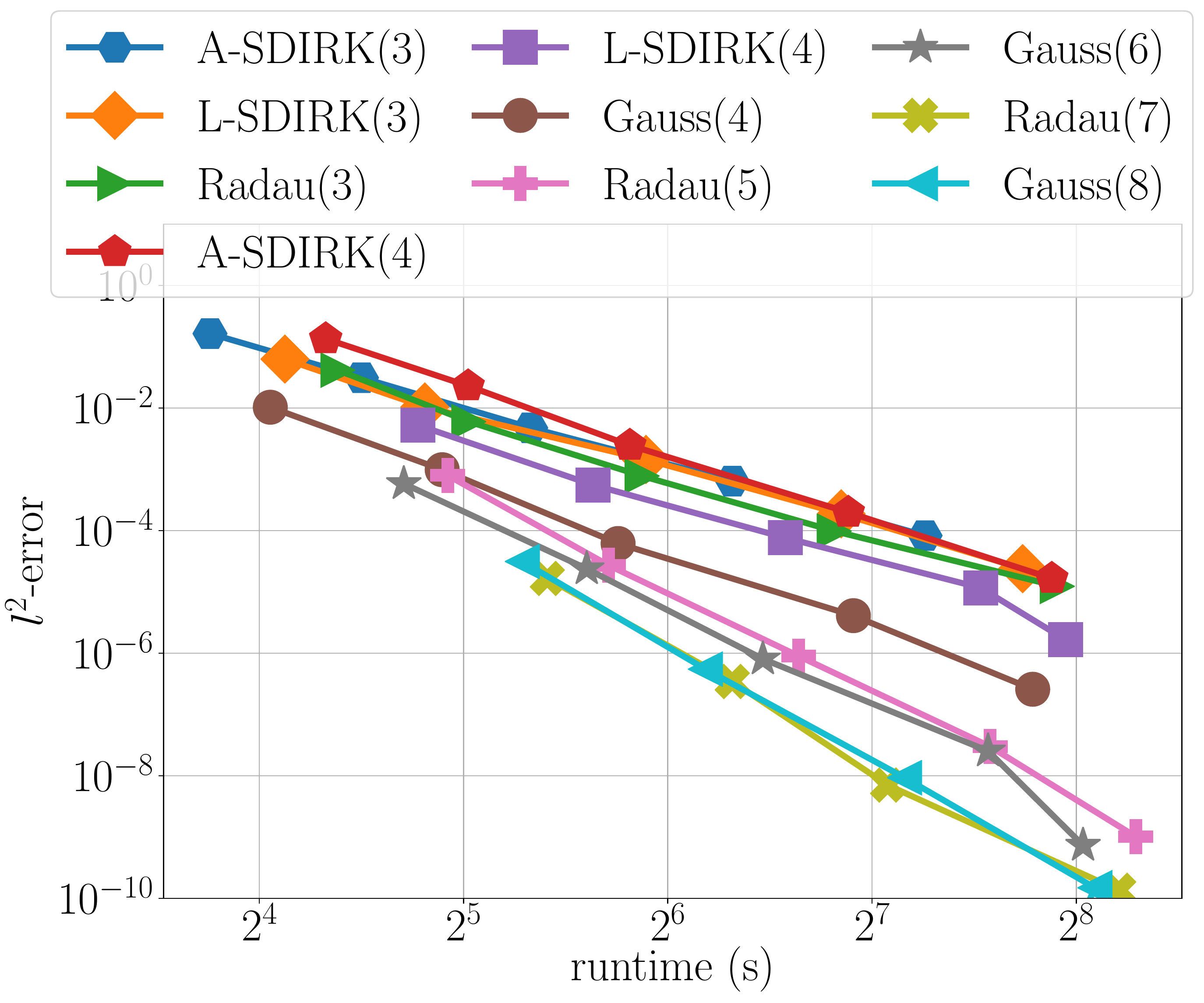}
  \end{subfigure}
   \begin{subfigure}[b]{0.475\textwidth}
    \includegraphics[width=\textwidth]{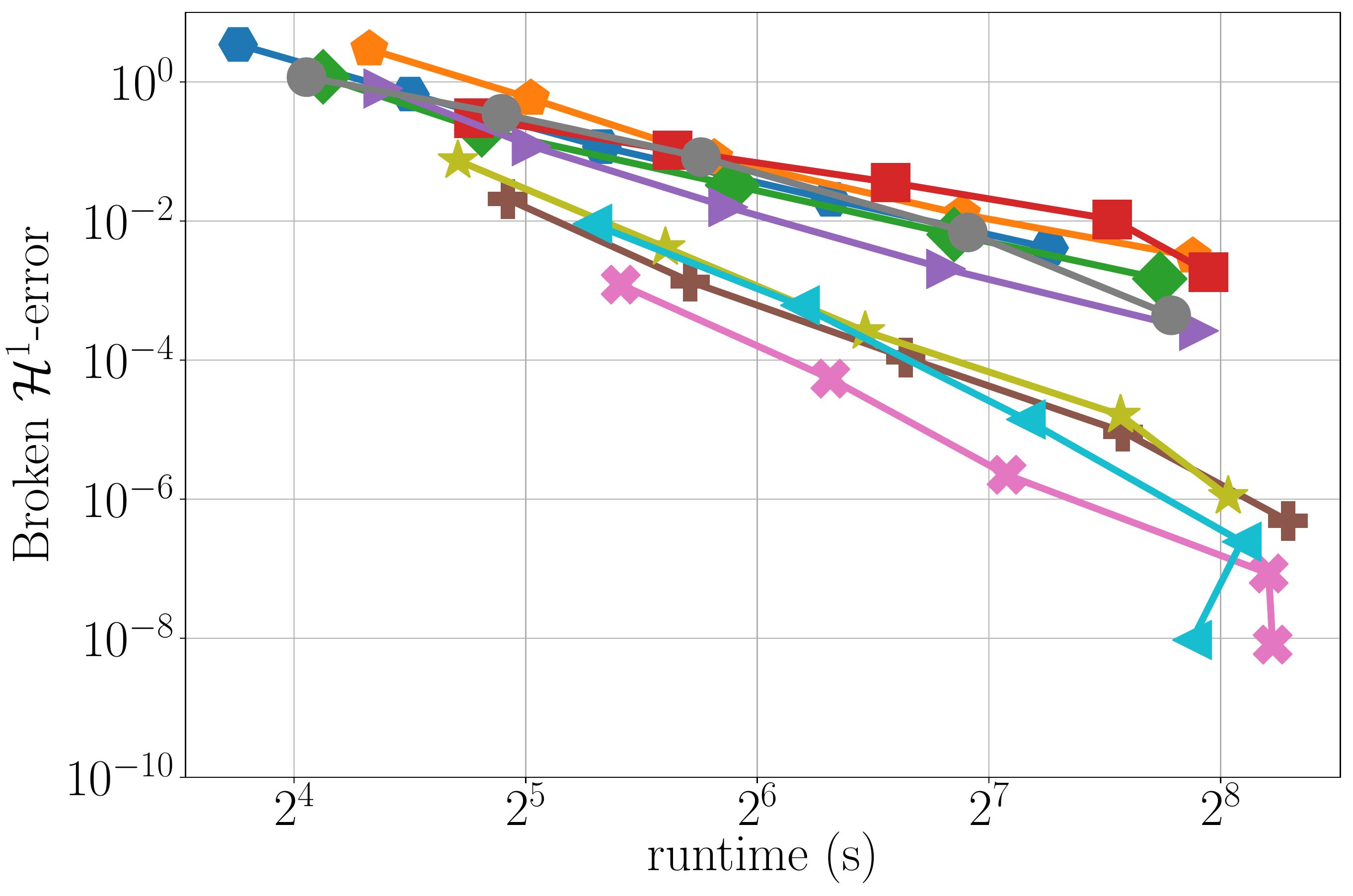}
  \end{subfigure}
  \\\vspace{2ex}
      \caption{$\ell^2$-error (right) and broken $\mathcal{H}^1$-error (left)
      at time $t=2$ as a function of wallclock time in seconds.}
      \label{fig:dg-time}
\end{figure}

}

\subsubsection{A hyperbolic example, and preconditioning with $\eta$ vs. $\gamma_*$}
\label{sec:numerics:dg:const}

The DG method is particularly well-suited for advection-dominated problems.
In the following subsections we vary $\varepsilon$ from $0$ (purely advective) to $0.01$.
The velocity field, initial condition, and numerical solution for $\varepsilon = 10^{-6}$
are shown in \Cref{fig:ad_advdiff}.

\begin{figure}[!htb]
  \centering
  \begin{subfigure}[b]{0.3\textwidth}
    \includegraphics[width=\textwidth]{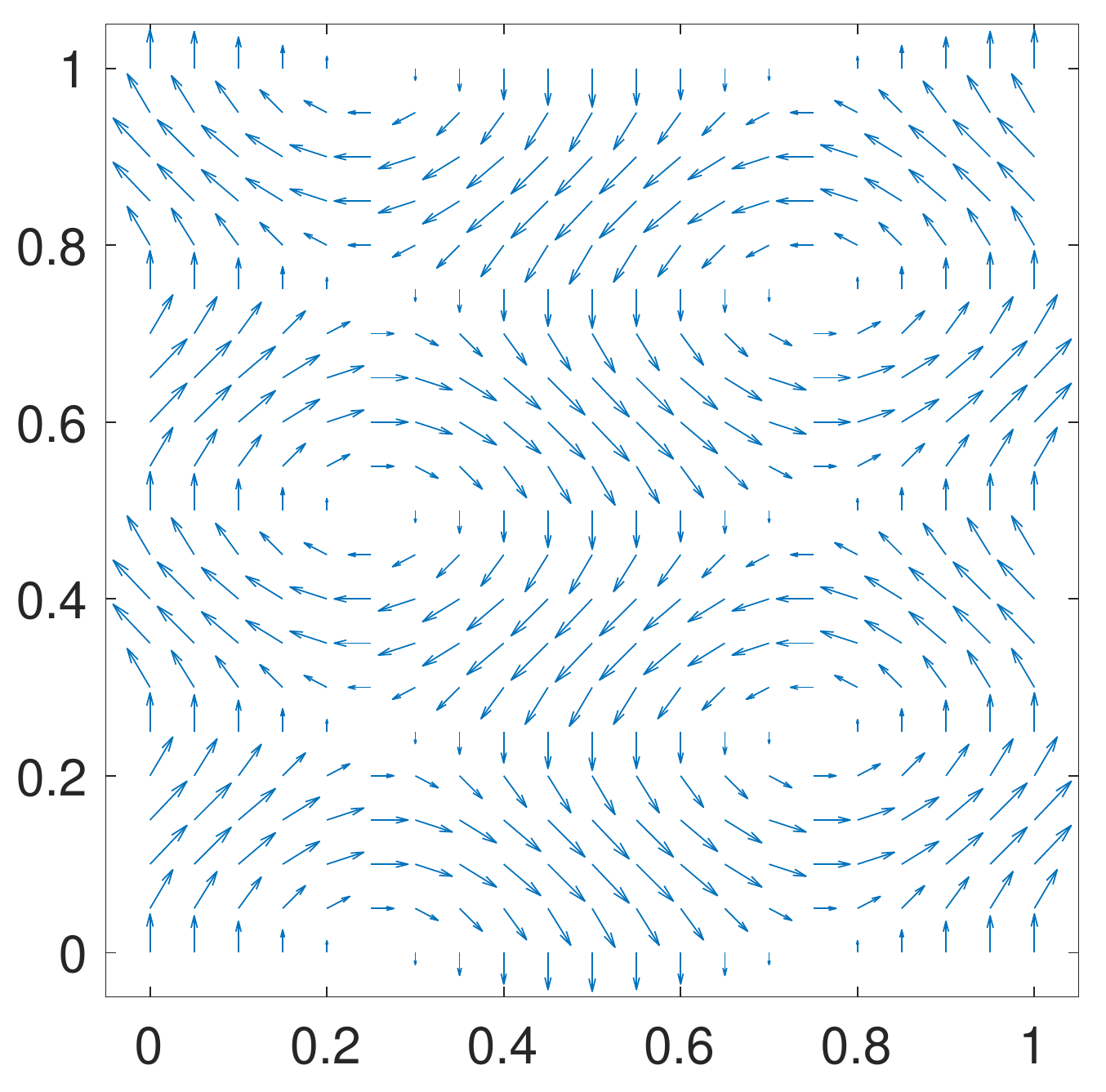}
    \caption{Velocity field}\label{fig:v}
  \end{subfigure}
   \begin{subfigure}[b]{0.3\textwidth}
    \includegraphics[width=\textwidth]{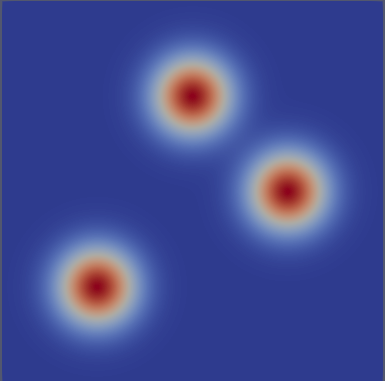}
    \caption{$t = 0$}
  \end{subfigure}
  \begin{subfigure}[b]{0.3\textwidth}
    \includegraphics[width=\textwidth]{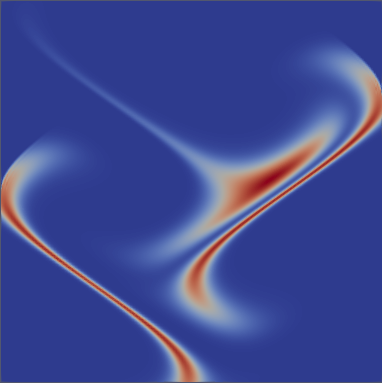}
    \caption{$t = 0.3$}
  \end{subfigure}
  \\\vspace{2ex}
   \begin{subfigure}[b]{0.3\textwidth}
    \includegraphics[width=\textwidth]{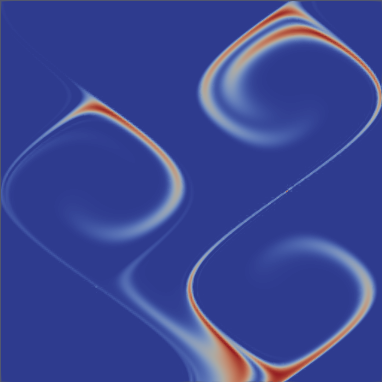}
    \caption{$t = 0.8$}
  \end{subfigure}
   \begin{subfigure}[b]{0.3\textwidth}
    \includegraphics[width=\textwidth]{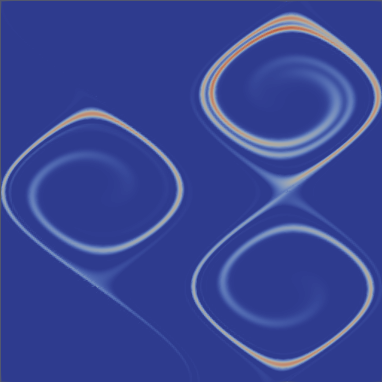}
    \caption{$t = 2.0$}
  \end{subfigure}
  \begin{subfigure}[b]{0.3\textwidth}
    \includegraphics[width=\textwidth]{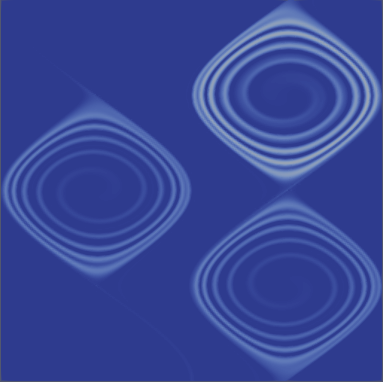}
    \caption{$t = 7.0$}
  \end{subfigure}
  \\\vspace{2ex}
      \caption{DG advection-diffusion problem with velocity field shown in
      subplot (a) and the solution plotted for various time points from
      $t=0$ to $t = 7.0$ in subplots (b-f). Heatmap indicates solution in the
      2d domain, with blue $\mapsto 0$ and red $\mapsto 1$.}
  \label{fig:ad_advdiff}
\end{figure}

First, we demonstrate the effectiveness of using $\gamma_*$ (\Cref{sec:solve:gamma})
instead of $\eta$ in the preconditioner, as well as the robustness of the proposed
method on a fully hyperbolic problem, where most papers have only discussed parabolic
PDEs. Thus, we set the diffusion coefficient $\varepsilon = 0$
and apply AIR as a preconditioner for individual systems
$(\gamma M - \delta t\mathcal{L})$.

AIR was originally designed for upwind DG discretizations of advection
and is well-suited for this problem. We use the \textit{hypre} implementation,
with distance 1.5 restriction with strength tolerance $\theta_R=0.01$, one-point
interpolation (type 100), Falgout coarsening (type 6) with strength tolerance
$\theta_C=0.1$, no pre-relaxation, and forward Gauss Seidel
post-relaxation (type 3), first on F-points followed by a second sweep on
all points. The domain is discretized using 4th-order finite elements on a
structured mesh, and the time step for each integration scheme is chosen
such that the spatial and temporal orders of accuracy match; for example,
for 8th-order integration we choose $\delta t = \sqrt{h}$, for mesh spacing
$h$, so that $\delta t^8 = h^4$. All linear systems are solved to a relative
tolerance of $10^{-12}$. There are a total of 1,638,400 spatial degrees-of-freedom
(DOFs), and the simulations are run on 288 cores on the Quartz machine at
Lawrence Livermore National Lab, resulting in $\sim$5600 DOFs/processor.

\Cref{tab:gamma} shows the average number of AIR iterations to solve for
each pair of stages of an IRK method using $\eta$ and $\gamma_*$ as the preconditioning
constants. Iteration counts are shown for Gauss, Radau IIA, and Lobatto IIIC integration,
with 2--5 stages, and the (factor of) reduction in iteration count achieved using $\gamma_*$
vs. $\eta$ is also shown. For 5-stage Lobatto IIIC integration, $\gamma_*$ yields
almost a $6\times$ reduction in total inner AIR iterations to solve for the
``hard'' stage ($\beta > \eta$), while in no cases is there an increase in
iteration count when using $\gamma_*$.

{

\renewcommand{\arraystretch}{1.15}
\begin{table}[!ht]
  \centering
  \begin{tabular}{|l || r | rr | rr | rrr |}  
  \hline
\multicolumn{9}{|c|}{Gauss}\\\hline
Stages/Order & 2/4 & \multicolumn{2}{c}{3/6} & \multicolumn{2}{|c}{4/8} & \multicolumn{3}{|c|}{5/10} \\\hline
Iterations$(\eta)$ & 17 & 6 & 30 & 11  & 47 & 8 & 16 & 70 \\
Iterations$(\gamma_*)$ & 11 & 6 & 15 & 10 & 19 & 8 & 13 & 23 \\\hline
Speedup & 1.5 & 1.0 & 2.0 & 1.1 & 2.5 & 1.0 & 1.2 & 3.0 \\\hline\hline
\multicolumn{9}{|c|}{Radau IIA}\\\hline
Stages/Order & 2/3 & \multicolumn{2}{c}{3/5} & \multicolumn{2}{|c}{4/7} & \multicolumn{3}{|c|}{5/9} \\\hline
Iterations$(\eta)$ & 12 & 5 & 39 & 11 & 64 & 8 & 16 & 97 \\
Iterations$(\gamma_*)$ & 12 & 5 & 18 & 9 & 21 & 8 & 12 & 25 \\\hline
Speedup & 1.0 & 1.0 & 2.2 & 1.2 & 3.0 & 1.0 & 1.3 & 3.9 \\\hline\hline
\multicolumn{9}{|c|}{Lobatto IIIC}\\\hline
Stages/Order & 2/2 & \multicolumn{2}{c}{3/4} & \multicolumn{2}{|c}{4/6} & \multicolumn{3}{|c|}{5/8} \\\hline
Iterations$(\eta)$ & 8 & 3 & 67 & 11 & 113 & 7 & 17 & 175 \\
Iterations$(\gamma_*)$ & 8 & 3 & 22 & 9 & 26 & 7 & 12 & 30 \\\hline
Speedup & 1.0 & 1.0 & 3.0 & 1.2 & 4.3 & 1.0 & 1.4 & 5.8 \\\hline\hline
  \end{tabular}
  \caption{Average AIR iterations to solve for each stage in an implicit
  Runge-Kutta method using preconditioners $(\eta M - \delta t \mathcal{L})^{-2}$ and
  $(\gamma_* M - \delta t \mathcal{L})^{-2}$, with $\gamma_*$ defined in \eqref{eq:gamma*}.
  The ratio of iterations$(\eta)$/iterations$(\gamma_*)$ is shown in the ``Speedup'' rows.}
  \label{tab:gamma}
\end{table}
}

\subsubsection{Diffusive problems and inner Krylov}\label{sec:numerics:dg:diff}

In \cite{Manteuffel:2019}, AIR was shown to be effective on some DG advection-diffusion
problems, and classical AMG is known to be effective on diffusion-dominated
problems. However, the region of comparable levels of advection and diffusion
remains the most difficult from a multigrid perspective. We use this to
demonstrate how methods developed here require a ``good'' preconditioner
for a backward Euler time step, $(\gamma M - \delta t \mathcal{L})^{-1}$,
in order to converge on more general IRK
methods. Fortunately, ensuring a preconditioner is sufficiently good can be
resolved by appealing to standard block preconditioning techniques, where an inner
iteration is used that applies multiple AIR iterations as a single preconditioner.

Here we consider an analogous problem to above, but set the diffusion coefficient
to $\varepsilon = 0.01$. We use a mesh with spacing $h \approx 0.001$, 2nd-order
DG finite elements, a time step of $\delta t = 0.1$, and three-stage 6th-order
Gauss integration. Altogether, this yields equal orders of accuracy, with time and
space error $\sim10^{-6}$. FGMRES \cite{saad1993flexible} is used for the
outer iteration, which allows for GMRES to be applied in an inner iteration
as a preconditioner for $(\gamma_* M - \delta t \mathcal{L})$. \Cref{fig:dg_o2} plots the
total number of AIR iterations per time step as a function of the number of
AIR iterations applied for each application of the preconditioner, using an
inner GMRES or an inner fixed-point (Richardson) iteration. An advection-dominated
problem with $\varepsilon = 10^{-6}$ is also shown for comparison.

Recall we have three stages, one of which is a single linear system corresponding
to a real eigenvalue, and the other corresponding to a pair of complex conjugate
eigenvalues, which we precondition as in \Cref{sec:solve}. The latter ends up being
the more difficult problem to solve -- for $\varepsilon =0.01$ (\Cref{fig:dg_o2_1e-2}),
the outer FGMRES iteration for the complex conjugate quadratic does not converge in
1000 iterations when using one AIR iteration as a preconditioner.
If two AIR iterations with GMRES are used as a preconditioner,
the FGMRES iteration converges in approximately 130 iterations, each of which
requires two applications of GMRES preconditioned with two AIR iterations,
yielding just over 500 total AIR iterations to converge. Further increasing
the number of AIR iterations per preconditioning yields nice convergence
using inner fixed-point or GMRES, with 150 and 112 total AIR iterations per
time step, respectively. In contrast, \Cref{fig:dg_o2_1e-6} shows that
additional AIR iterations for the advection-dominated case are generally
detrimental to overall computational cost (although the outer iteration
converges slightly faster, it does not make up for the additional linear
solves/iteration).

\begin{figure}[h!]
\centering

  \centering
  \begin{subfigure}[b]{0.475\textwidth}
	\includegraphics[width = \textwidth]{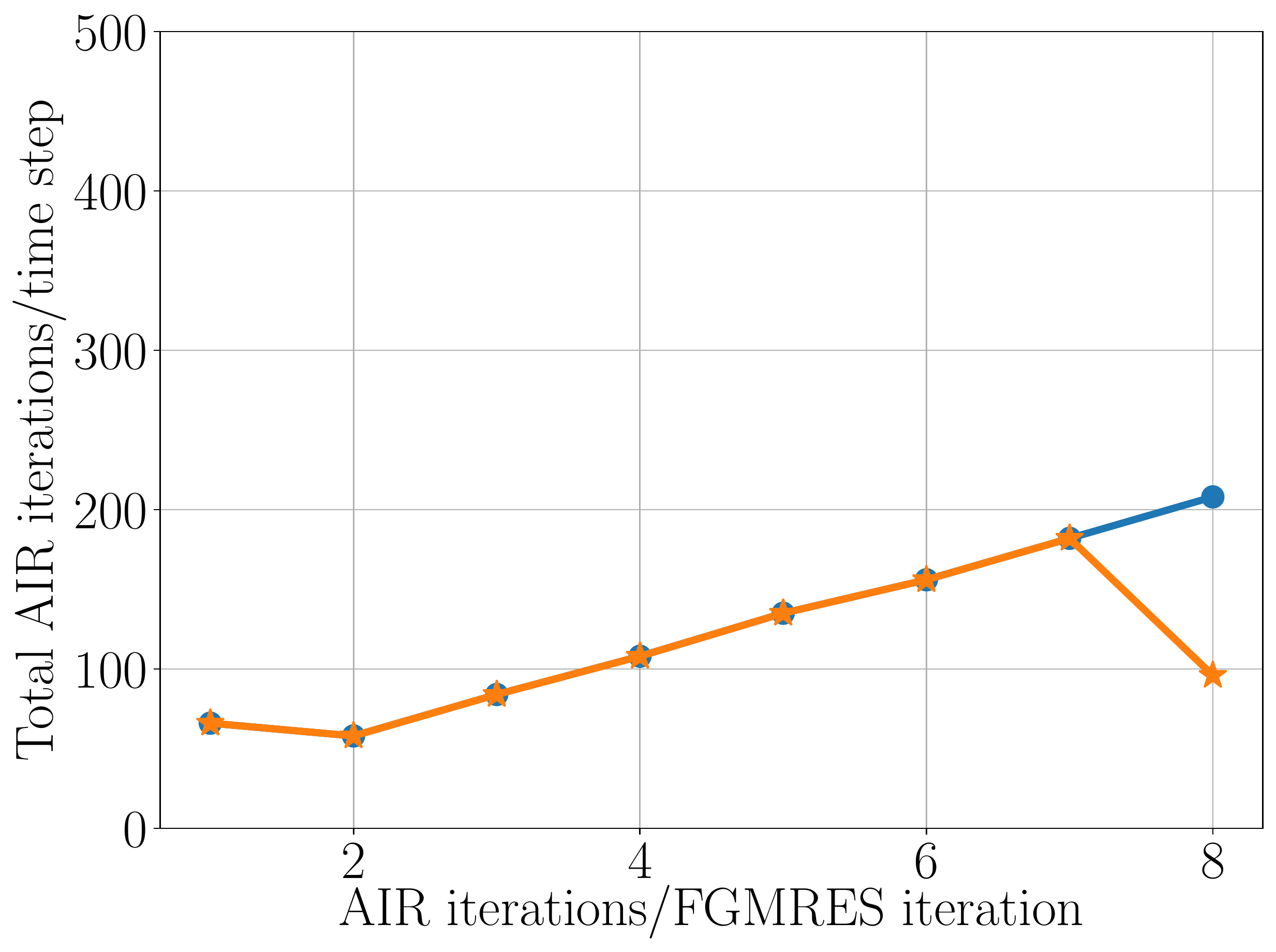}
	\caption{$\varepsilon = 10^{-6}$.}
	\label{fig:dg_o2_1e-6}
  \end{subfigure}
   \begin{subfigure}[b]{0.475\textwidth}
	\includegraphics[width = \textwidth]{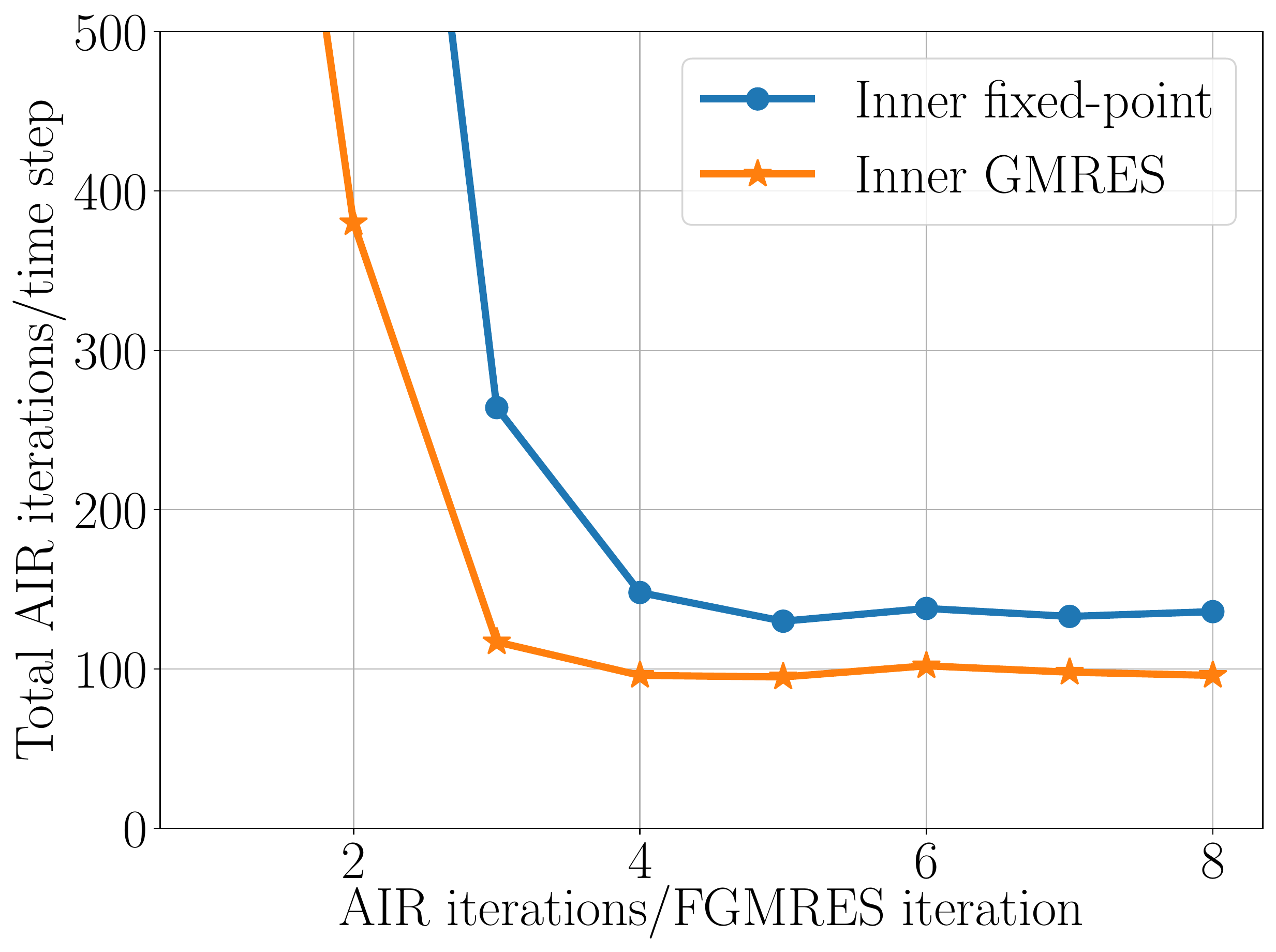}
	\caption{$\varepsilon = 0.01$.}
	\label{fig:dg_o2_1e-2}
  \end{subfigure}
\caption{Total AIR iterations per time step as a function of the number of
inner AIR iterations applied during each application of the preconditioner,
for diffusion coefficient $\varepsilon$.}
\label{fig:dg_o2}
\end{figure}

\subsubsection{Comparison of preconditioners}\label{sec:numerics:dg:comp}
{

We conclude this section by performing a detailed comparison with other
IRK solvers. We consider the three-stage and five-stage Gauss, Radau, and
Lobatto IRK integration schemes, with the velocity field shown in
\Cref{fig:ad_advdiff}, and diffusion coefficients
$\epsilon \in\{0, 10^{-4}, 0.01\}$. For these tests, we (somewhat 
arbitrarily) choose $\delta t = 0.1$, $h\approx 0.002$, and second-order
finite elements, resulting in
2,359,296 DOFs. All simulations are run on 256 cores on the Quartz
supercomputer at Lawrence Livermore National Lab, with 9,360 DOFs/core.
We compare the IRK methods developed here (denoted IRK), with the
GSL \cite{staff06} and LD \cite{rana2020new} block preconditioners
introduced earier, as well as a block ILU preconditioner as in 
\cite{pazner17}. As in \Cref{sec:numerics:dg:diff}, for the
case of $\epsilon =0.01$, we use three inner AIR iterations for
the IRK solver. We only use one inner AIR iteration for the GSL
and LD block preconditioners, as numerical tests indicated that
they neither need nor benefit from inner iterations. \Cref{tab:comp_prec}
presents the wallclock time and average iterations per time step to
take five time steps with each of the methods. Iterations are 
normalized as one preconditioner applied to each or all stages,
to account for the difference between ILU applied to the full IRK
operator and IRK/GSL/LD applied to individual stages, as well
as the use of three vs. one inner iteration. 

We highlight several observations:
\begin{itemize}
	\item In general, the method proposed here is faster than ILU
	in all cases but one (in which the times are very close), while
	offering as much as a $7\times$ speedup in one example (Radau(9)
	$\epsilon=0.01$). In general, ILU will be more competitive on
	a smaller number of MPI processes and less competitive on more
	MPI processes, due to its general degradation in parallel.

	\item The method proposed here is particularly effective on
	higher-order integrators (in this case, the 5-stage examples)
	and advection-dominated problems, in many cases offering a
	$10-15\times$ speedup over the block preconditioning from
	\cite{staff06}. The very new method from \cite{rana2020new}
	is more competitive, typically only $2-5\times$ slower on the
	more advective problems, while performing better than the
	method developed here for all cases of $\epsilon = 0.01$.
	The better performance for $\epsilon = 0.01$ is because the
	method from \cite{rana2020new} appears more robust to not
	having an accurate inner inverse, whereas IRK requires three
	inner AIR iterations for good performance on $\epsilon = 0.01$
	(thus, costing $3\times$ as much for one true
	iteration). However, this property is a blessing and a curse:
	numerical tests also indicated that additional inner iterations
	did \emph{not} improve convergence of \cite{staff06,rana2020new}
	on advective problems, thus limiting their performance to that of
	the outer block preconditioner, which currently lacks robust
	theoretical support. In contrast, the IRK solvers developed
	here are \emph{guaranteed} to be robust, as long as one provides
	a reasonably accurate inner inverse.
\end{itemize}
{

\renewcommand{\arraystretch}{1.15}
\begin{table}[!ht]
  \centering
  \begin{tabular}{| c | l | cccc |}  
  \hline
\parbox[t]{2mm}{\multirow{12}{*}{\rotatebox[origin=c]{90}{3 stages}}}
& \textbf{Gauss(6)} & GSL \cite{staff06} & LD \cite{rana2020new} & ILU \cite{Pazner2019a} & IRK \\\cline{2-6}
& $\epsilon = 0$ & 7.76 (66) & 4.46 (27) & 3.36 (45) & \textbf{3.24} (24) \\ 
& $\epsilon = 10^{-4}$ & 16.4 (67) & 8.85 (28) & \textbf{5.76} (73) & 6.02 (20) \\
& $\epsilon = 0.01$ & 10.5 (51) & \textbf{7.48} (28) & 85 (491) & 14.7 (44)
\\\cline{2-6}
& \textbf{Radau(5)} & GSL \cite{staff06} & LD \cite{rana2020new} & ILU \cite{Pazner2019a} & IRK \\\cline{2-6}
& $\epsilon = 0$ & 11.9 (101) & 5.35 (32) & 5.20 (68) & \textbf{3.81} (27) \\
& $\epsilon = 10^{-4}$ & 25.1 (101) & 10.62 (34) & 7.11 (87) & \textbf{6.42} (22) \\
& $\epsilon = 0.01$ & 17.6 (85) & \textbf{10.32} (41) & 164 (716) & 31.0 (94)
\\\cline{2-6}
&\textbf{Lobatto(4)} & GSL \cite{staff06} & LD \cite{rana2020new} & ILU \cite{Pazner2019a} & IRK \\\cline{2-6}
& $\epsilon = 0$ & 19.3 (168) & 7.86 (47.2) & 10.4 (118.6) & \textbf{4.89} (36) \\
& $\epsilon = 10^{-4}$ & 39.2 (167) & 15.2 (49.2) & 9.43 (109.2) & \textbf{7.91} (25) \\
& $\epsilon = 0.01$ & 26.6 (133) & \textbf{14.3} (57) & 275 (955.8) & 38.9 (115)
\\\hline\hline
\parbox[t]{2mm}{\multirow{12}{*}{\rotatebox[origin=c]{90}{5 stages}}}
& \textbf{Gauss(10)} & GSL \cite{staff06} & LD \cite{rana2020new} & ILU \cite{Pazner2019a} & IRK \\\cline{2-6}
& $\epsilon = 0$ & 24.9 (151) & 11.5 (43) & 12.7 (78) & \textbf{4.40} (21) \\
& $\epsilon = 10^{-4}$ & 55.9 (153) & 25.8 (52) & 11.6 (73) & \textbf{10.1} (22) \\
& $\epsilon = 0.01$ & 28.4 (91) & \textbf{17.9} (41) & 134 (449) & 22.6 (42)
\\\cline{2-6}\cline{2-6}
& \textbf{Lobatto(8)} & GSL \cite{staff06} & LD \cite{rana2020new} & ILU \cite{Pazner2019a} & IRK \\\cline{2-6}
& $\epsilon = 0$ & 70.6 (435) & 21.7 (83) & 113 (403) & \textbf{5.41} (44) \\
& $\epsilon = 10^{-4}$ & 161 (444) & 47.7 (99) & 15.9 (94) & \textbf{11.2} (25) \\
& $\epsilon = 0.01$ & 73.0 (242) & \textbf{34.3} (84) & 281 (703) & 52.3 (100)
\\\cline{2-6}\cline{2-6}
& \textbf{Radau(9)} & GSL \cite{staff06} & LD \cite{rana2020new} & ILU \cite{Pazner2019a} & IRK \\\cline{2-6}
& $\epsilon = 0$ & 40.2 (241) & 14.6 (53) & 25.0 (140) & \textbf{4.72} (23) \\
& $\epsilon = 10^{-4}$ & 89.5 (243) & 31.3 (64) & 13.8 (84) & \textbf{10.3} (23) \\
& $\epsilon = 0.01$ & 47.8 (153) & \textbf{25.6} (61) & 220 (609) & 33.4 (64)
\\\hline\hline
  \end{tabular}
  \caption{Comparison of wallclock time (left) to take 5 time steps, and
  average iteration count per time step (normalized
  as one preconditioner applied to all stages, right in ($\cdot$)) for various
  IRK preconditioners. Fastest wallclock times for each fixed $\epsilon$ are
  shown in bold.}\label{tab:comp_prec}
\end{table}

}

}

\subsection{High-order matrix-free discretization of diffusion}

In this example, we illustrate the use of high-order IRK methods coupled with high-order finite element spatial discretizations.
It is well-known that matrix assembly becomes prohibitively expensive for high-order finite elements.
Naive algorithms typically require $\mathcal{O}(p^{3d})$ operations to assemble the resulting system matrix, where $p$ is the polynomial degree and $d$ is the spatial dimension.
Techniques such as sum factorization can reduce this cost on tensor-product elements to $\mathcal{O}(p^{2d+1})$, however this cost can still be prohibitive for large values of $p$ \cite{Melenk2001}.
On the other hand, matrix-free operator evaluation on tensor-product meshes can be performed in $\mathcal{O}(p^{d+1})$ operations \cite{Orszag1980}, motivating the development of solvers and preconditioners that can be constructed and applied without access to the assembled system matrix \cite{Kronbichler2019}.

We consider a high-order finite element discretization of the linear heat equation on spatial domain $\Omega$,
\[
	\int_\Omega \partial_t (u_h) v_h \, dx + \int_\Omega \nabla u_h \cdot \nabla v_h \, dx = \int_\Omega f v_h \, dx,
\]
where $u_h, v_h \in V_h$, and $V_h$ denotes the degree-$p$ $H^1$-conforming finite element space defined on a mesh $\mathcal{T}$ consisting of tensor-product elements (i.e.\ quadrilaterals or hexahedra).
The matrix-free action of the corresponding operator is computed in $\mathcal{O}(p^{d+1})$ operations using the \textit{partial assembly} features of the MFEM finite element library \cite{Anderson2020}.
In order to precondition the resulting system, we make use of a low-order refined preconditioner, whereby the high-order system is preconditioned using a spectrally equivalent low-order finite element discretization computed on a refined mesh \cite{Canuto2010}.
The low-order refined discretization can be assembled in $\mathcal{O}(p^d)$ time, thereby avoiding the prohibitive costs of high-order matrix assembly.
We make use of the uniform preconditioners for the low-order refined problem based on subspace corrections, developed in \cite{Pazner2019a}.

For this test case, take the spatial domain to be $\Omega = [0,1] \times [0,1]$, with periodic boundary conditions.
We choose the forcing term
\[
	f(x, y, t) = \sin (2\pi x)\cos(2\pi y) \left(\cos(t) + 8 \pi^2 (2 + \sin(t)) \right),
\]
which corresponds to the exact solution
\[
	u(x, y, t) = \sin(2\pi x)\cos(2\pi y)(2 + \sin(t)).
\]
We begin with a very coarse $3 \times 3$ mesh, and integrate in time until $t=0.1$ using the Gauss and Radau IIA methods of orders 2 through 10.
For each test case, the finite element polynomial degree is set to $k-1$, where $k$ is the order of accuracy of the time integration method, resulting in $k$th order convergence in both space and time.
The mesh and time step are refined by factors of two to confirm the high-order convergence in space and time of the method.
The relative $L^2$ error, obtained by comparing against the exact solution, is shown in \Cref{fig:high-order-diff-errors}.

\begin{figure}[!ht]
	\centering
	\includegraphics{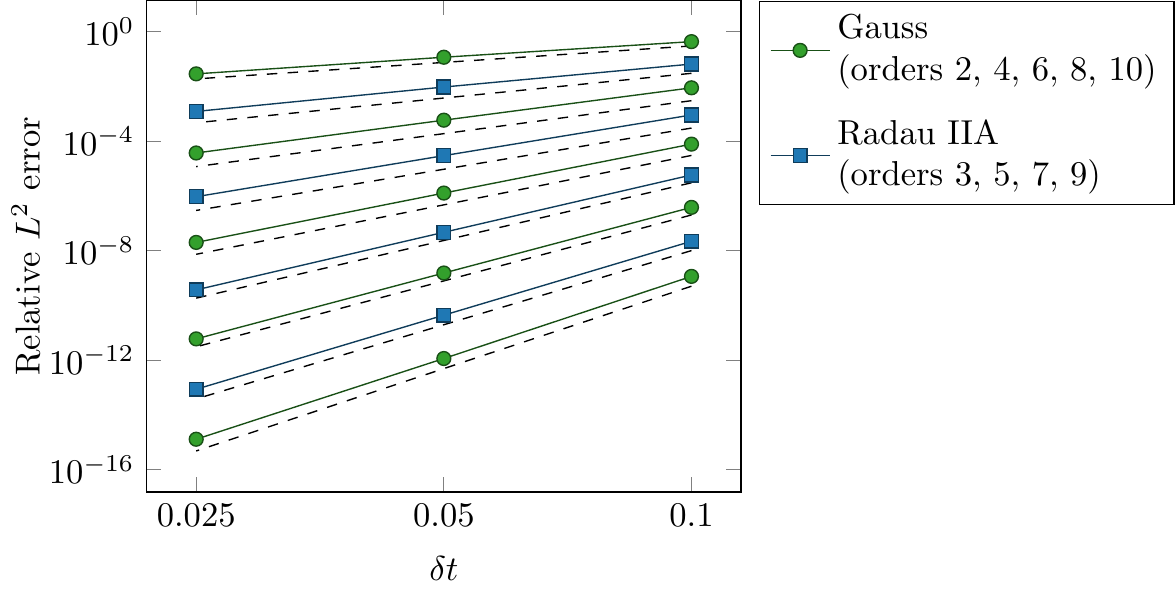}
	\caption{
		High-order convergence in space and time for the matrix-free diffusion problem.
		Gauss and Radau IIA methods of orders 2 through 10 are used.
		The dashed lines indicate the expected rates of convergence for each method.
	}
	\label{fig:high-order-diff-errors}
\end{figure}

We also use this test case to study the effect of inner iterations on the convergence of the iteration solver.
As discussed in \Cref{sec:inexact-precond}, it is important that the underlying preconditioner provides a good approximation of the inverse of the operator.
For that reason, we consider the use of an inner Krylov solver at every iteration.
Since this corresponds to using a variable preconditioner at each iteration, a flexible Krylov method may have to be used for the outer iteration, although in practice good convergence is often still observed using the standard CG method \cite{Notay2000}.
In particular, we compare the total number of preconditioner applications required to converge the outer iteration to a relative tolerance of $10^{-10}$, both with and without an inner Krylov solver.
For the inner Krylov solver, we use a CG iteration with the same relative tolerance as the outer iteration in order to give a good approximation to the inverse of the operator.
The iteration counts are displayed in \Cref{fig:high-order-diff-iters}.
We note that for the fully implicit IRK methods, using an inner Krylov solver can reduce the total number of preconditioner applications by about a factor of 1.5, although this depends on the type of method and order of accuracy.
As expected, the use of inner iterations does not reduce the total number of preconditioner applications for DIRK methods.
In addition, for this test case, the total number of preconditioner applications required for the second and fourth order Gauss IRK methods is between 1.3 and 2 times smaller than those required for the corresponding equal-order DIRK methods.

\begin{figure}[!ht]
	\centering
	\includegraphics{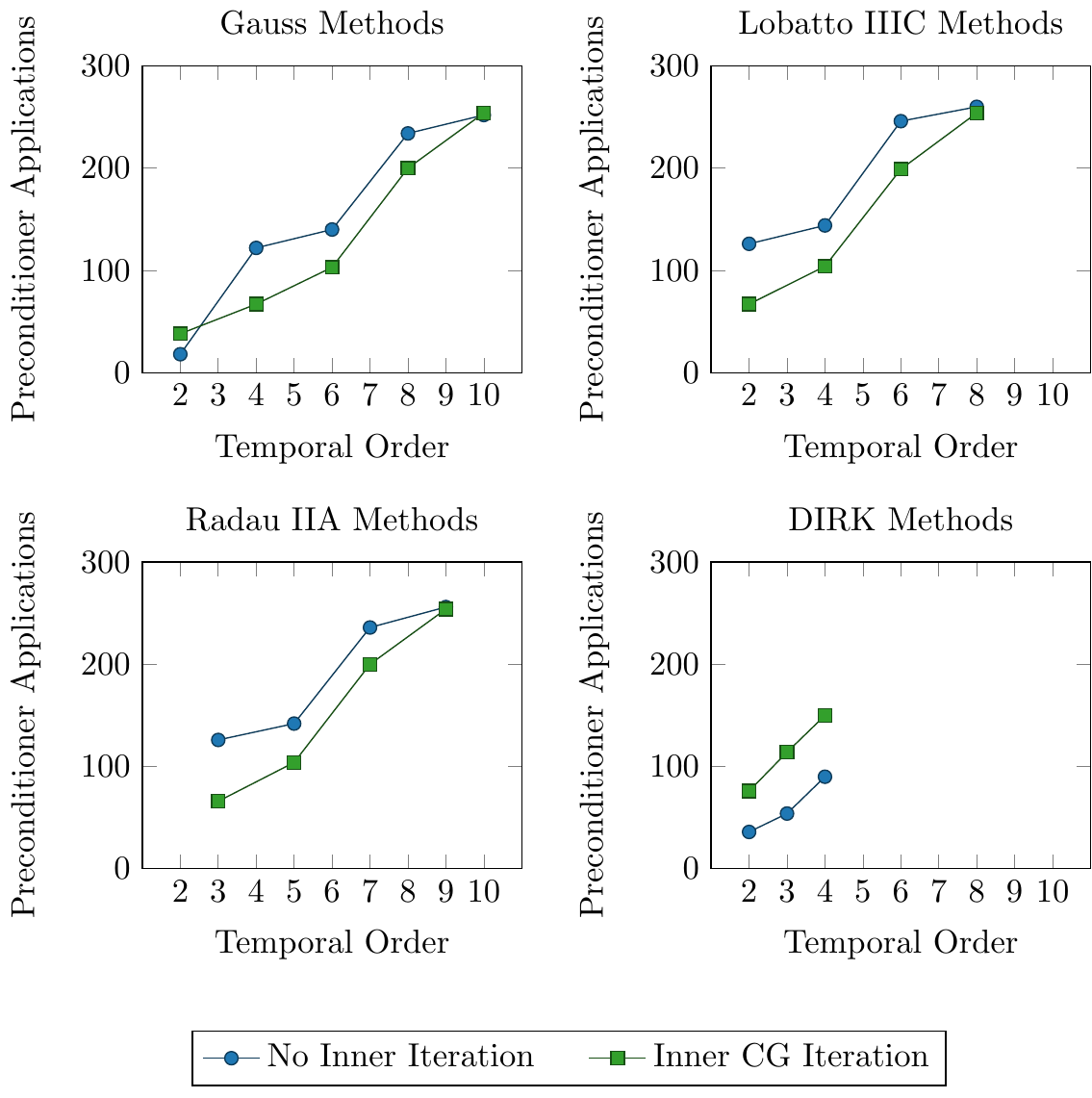}
	\caption{
		Comparison of total number of preconditioner applications with and without
		inner iterations.
		Both the outer iteration and the inner CG iteration are converged to a
		relative tolerance of $10^{-10}$.
	}
	\label{fig:high-order-diff-iters}
\end{figure}

\section{Conclusions}\label{sec:conc}

This paper introduces a theoretical and algorithmic framework for the fast, parallel
solution of fully implicit Runge-Kutta and DG discretizations in time for linear
numerical PDEs. Theory is developed to guarantee the {preconditioned condition
number is bounded by a small, order-one constant} under fairly
general assumptions on the spatial discretization that yield stable time integration.
Numerical results demonstrate the new method on various high-order finite-difference
and finite-element discretizations of linear parabolic and hyperbolic problems,
demonstrating fast, scalable solution of up to 10th order accuracy. In several
cases, the new method can achieve 4th-order accuracy using Gauss integration
with roughly half the number of preconditioner applications as required using
standard SDIRK techniques, {and in many cases the method outperforms previous state
of the art.} Ongoing work involves addressing fully nonlinear
problems and algebraic constraints, in particular, without assuming that the
linear system \eqref{eq:k0} can be expressed in Kronecker-product form (thus
allowing for a true Newton or better Newton-like method compared with the
commonly used/analyzed simplified Newton approach).

\appendix
\section{Proofs}\label{appendix}

\begin{theorem}\label{th:cond}
Suppose \Cref{ass:eig,ass:fov} hold, that is, $\eta > 0$ and $W(\cL) \leq 0$,
and suppose $\cL$ is real-valued. Let $\mathcal{P}_{\delta,\gamma}$ denote
the preconditioned operator as in \eqref{eq:P_gen},
where $[(\eta I  - \cL)^2 + \beta^2 I]$ is preconditioned with
$(\delta I - \cL)^{-1} (\gamma I - \cL)^{-1}$, for $\delta, \gamma \in (0, \infty)$.

Let $\kappa({\cal P}_{\delta, \gamma})$ denote the two-norm condition number of ${\cal P}_{\delta,\gamma}$,
and define $\gamma_*$ by
\begin{align} \label{eq:gamma*_gen}
\gamma_* \coloneqq \frac{\eta^2+\beta^2}{\delta}.
\end{align}
Then
\begin{align} \label{eq:kappa_gen_gamma*_bound}
\kappa(\mathcal{P}_{\delta, \gamma_*}) \leq \frac{1}{2 \eta} \left( \delta + \frac{\eta^2 + \beta^2}{\delta} \right).
\end{align}

Moreover, (i) bound \eqref{eq:kappa_gen_gamma*_bound} is tight {when considered over
all $\cL$ that satisfy \Cref{ass:fov}} in the sense that $\exists$ $\cL$ such that
\eqref{eq:kappa_gen_gamma*_bound} holds with equality, and (ii) $\gamma = \gamma_*$ is optimal
in the sense that, without further assumptions on $\cL$, $\gamma_*$ minimizes a tight
upper bound on $\kappa({\cal P}_{\delta, \gamma})$, with {$\gamma_* =
\argmin_{\gamma\in(0,\infty)} \max_{\cL} \kappa({\cal P}_{\delta, \gamma})$}.
\end{theorem}

\begin{proof}
We prove this theorem via a sequence of three lemmas, successively proving the
upper bound in \eqref{eq:kappa_gen_gamma*_bound}, followed by the tightness of
this bound, followed by the optimality of $\gamma_*$.

\begin{lemma}[Upper bound]
Under the assumptions of \Cref{th:cond},
$\kappa(\mathcal{P}_{\delta, \gamma_*}) \leq
\frac{1}{2 \eta} \left( \delta + \frac{\eta^2 + \beta^2}{\delta} \right)$
\eqref{eq:kappa_gen_gamma*_bound}.
\end{lemma}
\begin{proof}
The square of the condition number of ${\cal P}_{\delta, \gamma}$ is given by

\begin{align}
\label{eq:kappa_gen_def}
\kappa^2({\cal P}_{\delta,\gamma})
=
\Vert {\cal P}_{\delta,\gamma} \Vert^2
\Vert {\cal P}_{\delta,\gamma}^{-1} \Vert^2
=
\underset{{\bm{v} \neq 0}}{\max} \frac{\Vert {\cal P}_{\delta,\gamma} \bm{v} \Vert^2 }{\Vert \bm{v} \Vert^2}
\frac{1}{\displaystyle{\underset{{\bm{v} \neq 0}}{\min} \frac{\Vert {\cal P}_{\delta,\gamma} \bm{v} \Vert^2 }{\Vert \bm{v} \Vert^2}}},
\end{align}

where for real-valued $\cL$, the max and min can be obtained by restricting ourselves
to real-valued $\bm{v}$. The key step in establishing \eqref{eq:kappa_gen_gamma*_bound} is bounding
$\Vert {\cal P}_{\delta,\gamma} \Vert^2$ and $\Vert {\cal P}_{\delta,\gamma}^{-1} \Vert^2$ from above
by bounding $\Vert {\cal P}_{\delta,\gamma} \bm{v} \Vert^2 / \Vert \bm{v} \Vert^2$
from above and below, respectively.

Consider the form of the preconditioned operator ${\cal P}_{\delta,\gamma}$ in
\eqref{eq:P_gen} and make the substitution $\bm{v} \mapsto (\gamma I - \cL) (\delta I - \cL) \bm{w}$.
Using the fact that rational functions of $\mathcal{L}$ commute,
$\|{\cal P}_{\delta,\gamma} \bm{v}\|^2$ can be expanded for real-valued
$\bm{v}$ (and, thus, real-valued $\bm{w}$) as

\begin{align}
\label{eq:norm_Pv_gen}
\begin{split}
\Vert {\cal P}_{\delta, \gamma} \bm{v} \Vert^2 &=
	\bVert [( \eta I - \cL)^2 + \beta^2] \bm{w} \bVert^2, \\
& = \bVert [(\eta^2+\beta^2)\bm{w} - 2\eta\cL\bm{w} + \cL^2\bm{w} \bVert^2 \\
&= \bVert (\eta^2 + \beta^2) \bm{w} + \cLs \bm{w} \bVert^2
	-4 \eta (\eta^2 +\beta^2) \langle \cL \bm{w}, \bm{w} \rangle
	\\&\quad\quad
	-4 \eta \langle \cL (\cL \bm{w}), \cL \bm{w} \rangle
	+4 \eta^2 \bVert \cL \bm{w} \bVert^2.
\end{split}
\end{align}

Similarly, expanding $\|\bm{v}\|^2$ yields

\begin{align}
\label{eq:norm_v_gen}
\begin{split}
\Vert \bm{v} \Vert^2
&= \bVert ( \gamma I - \cL) ( \delta I - \cL) \bm{w} \bVert^2, \\
&= \bVert \delta\gamma \bm{w} -(\delta+\gamma)\cL\bm{w} + \cL^2\bm{w} \bVert^2, \\
&= \bVert \delta \gamma \bm{w} + \cLs \bm{w} \bVert^2
	- 2 \delta \gamma (\delta + \gamma ) \langle \cL \bm{w}, \bm{w} \rangle
	\\&\quad\quad
	-2  (\delta + \gamma ) \langle \cL ( \cL \bm{w}), \cL \bm{w} \rangle
	+  (\delta + \gamma)^2 \bVert \cL \bm{w} \bVert^2.
\end{split}
\end{align}

Thus, the key ratio in \eqref{eq:kappa_gen_def} takes the form
\begin{align}
\label{eq:P_gen_frac}
\frac{\Vert {\cal P}_{\delta,\gamma} \bm{v} \Vert^2}{\Vert \bm{v} \Vert^2}
=
\frac{
c_0(\bm{w}) f_0(\bm{w}) + c_1 f_1(\bm{w}) + c_2 f_2(\bm{w}) + c_3 f_3(\bm{w})
}
{
f_0(\bm{w}) + f_1(\bm{w}) + f_2(\bm{w}) + f_3(\bm{w})
},
\end{align}
where for $\delta, \gamma > 0$, we have defined the functions and constants
\begin{equation}
\begin{aligned}
\label{eq:f_c_gen_def}
f_0 &\coloneqq \bVert \delta \gamma \bm{w} + \cLs \bm{w} \bVert^2 \geq 0,
\quad
&&c_0\coloneqq \frac{\bVert (\eta^2 + \beta^2) \bm{w} + \cLs \bm{w} \bVert^2}{\bVert \delta \gamma \bm{w} + \cLs \bm{w} \bVert^2} \geq 0,
\\
f_1 &\coloneqq -2 \delta \gamma (\delta + \gamma)  \langle \cL \bm{w}, \bm{w} \rangle \geq 0,
\quad
&&c_1\coloneqq \frac{\eta^2 + \beta^2}{\delta \gamma} \frac{2 \eta}{\delta + \gamma} > 0,\\
f_2 &\coloneqq -2( \delta + \gamma) \langle \cL( \cL \bm{w}), \cL \bm{w} \rangle \geq 0,
\quad
&&c_2\coloneqq \frac{2 \eta}{\delta + \gamma} > 0,\\
f_3 &\coloneqq (\delta+\gamma)^2 \bVert \cL \bm{w} \bVert^2 \geq 0,
\quad
&&c_3\coloneqq \left(\frac{2 \eta}{\delta + \gamma}\right)^2 > 0.
\end{aligned}
\end{equation}
Note that functions $f_1$ and $f_2$ are non-negative by
assumption of $W( \cL ) \leq 0$, while for all $\bm{w}\neq\mathbf{0}$, it must hold that
either $c_0f_0 > 0$ or $c_3f_3 > 0$ (or both, because $c_3f_3 = 0$ i.f.f.
$\cL\bm{w} = \mathbf{0}$, which implies $c_0f_0 > 0$ for $\bm{w}\neq\mathbf{0}$).

Since all of the addends in the numerator and denominator of \eqref{eq:P_gen_frac} are non-negative, and at least one addend in each is positive, \eqref{eq:P_gen_frac} can be bounded as
\begin{align*}
\min \{ c_0, c_1, c_2, c_3 \} \eqqcolon c_{\min}
\leq
\frac{\Vert {\cal P}_{\delta,\gamma} \bm{v} \Vert^2}{\Vert \bm{v} \Vert^2}
\leq c_{\max}
\coloneqq \max \{ c_0, c_1, c_2, c_3 \}.
\end{align*}
Applying these bounds to the norms in \eqref{eq:kappa_gen_def} yields
\begin{align} \label{eq:P_gen_bounds}
\Vert {\cal P}_{\delta, \gamma} \Vert \leq \sqrt{c_{\max}},
\quad
\Vert {\cal P}_{\delta, \gamma}^{-1} \Vert \leq \frac{1}{\sqrt{c_{\min}}}.
\end{align}

Bounding $c_0$ for general $\gamma$, and hence $c_{\min}$ and $c_{\max}$, is
difficult because the sign of $\langle \cLs \bm{w}, \bm{w} \rangle$ (which
appears in expanding $\bVert \delta \gamma \bm{w} + \cLs \bm{w} \bVert^2$) is not
known for general $\cL$, noting that the sign of $W(\cL)$ does not determine
that of $W(\cLs)$. However, observe from \eqref{eq:f_c_gen_def} that the judicious
choice of $\gamma = \gamma_* \coloneqq (\eta^2 + \beta^2)/\delta$ yields $c_0(\bm{w}) = 1$.
Moreover, in the final part of this proof we demonstrate that
$\gamma = \gamma_*$ is optimal, and, as such, moving forward we
only consider the case $\gamma = \gamma_*$.

Letting $\gamma = \gamma_* \coloneqq (\eta^2 + \beta^2)/\delta$, from
\eqref{eq:f_c_gen_def} one has $c_0 = 1 \geq c_1 = c_2 = \sqrt{c_3} = 2
\eta/(\delta + \gamma_*)$, where the inequality $1 \geq 2 \eta/(\delta +
\gamma_*)$ follows by noting the equivalent relation $\delta^2
-2\eta\delta+\eta^2+\beta^2 \geq 0$ for all $\eta,\delta>0$. Thus, for $\gamma =
\gamma_*$, the bounds in \eqref{eq:P_gen_bounds} are given by
\begin{align}
\label{eq:P_gen_gamma*_bounds}
\Vert {\cal P}_{\delta,\gamma_*} \Vert \leq 1,
\quad
\Vert {\cal P}_{\delta,\gamma_*}^{-1} \Vert
\leq \frac{\delta + \gamma_*}{2 \eta}
= \frac{1}{2 \eta} \left( \delta + \frac{\eta^2 + \beta^2}{\delta} \right)
\end{align}
Applying these bounds to the condition number \eqref{eq:kappa_gen_def}
yields the upper bound in \eqref{eq:kappa_gen_gamma*_bound}.
\end{proof}

We now show that bound \eqref{eq:kappa_gen_gamma*_bound} is tight.
We do so by construction, showing that the
bound in \eqref{eq:kappa_gen_gamma*_bound} is achieved for
certain matrices that satisfy \Cref{ass:fov}.
\begin{lemma}[Tightness]
$\exists$ $\cL$ such that \eqref{eq:kappa_gen_gamma*_bound} holds with equality.
\end{lemma}
\begin{proof}
Note that the min/max
of $\Vert {\cal P}_{\delta,\gamma} \bm{v} \Vert^2/\Vert \bm{v} \Vert^2$ over $\bm{v}$
for real-valued ${\cal P}_{\delta,\gamma}$ is equivalent when minimizing over real
or complex $\bm{v}$; we now consider complex $\bm{v}$ for theoretical
purposes. To that end, let $~\bm{v}=(\gamma I - \cL)(\delta I - \cL) \bm{w}$, but
suppose that $(\textrm{i} \xi, \bm{w})$ is an eigenpair of $\cL$, with $\xi$
a real number and $\bm{w}$ a complex eigenvector. Plugging into
$\| {\cal P}_{\delta,\gamma} \bm{v}\|^2$ \eqref{eq:norm_Pv_gen} and $\|\bm{v}\|^2$
\eqref{eq:norm_v_gen}, and taking the ratio as in \eqref{eq:P_gen_frac},
define the following function of $\xi$:

\begin{align} \label{eq:H_gen_def}
{\cal H}_{\delta,\gamma}(\xi)
\coloneqq
\left. \frac{\Vert {\cal P}_{\delta, \gamma} \bm{v} \Vert^2 }{\Vert \bm{v} \Vert^2} \right|_{\cL \bm{w} = \mathrm{i} \xi \bm{w}}
=
\frac{|(\eta- \mathrm{i} \xi)^2 + \beta^2|^2}{|(\delta \gamma - \xi^2 - \mathrm{i} (\delta + \gamma) \xi|^2}
=
\frac{(\delta \gamma_* - \xi^2)^2 + (2 \eta \xi)^2}{(\delta \gamma-\xi^2)^2 + [\xi ( \delta + \gamma) ]^2},
\end{align}
where we have made use of $\delta \gamma_* = \eta^2 + \beta^2$.
By virtue of restricting that $\bm{w}$ be an eigenvector, from \eqref{eq:kappa_gen_def} we have
\begin{align}
\label{eq:H_gen_bounds}
\frac{1}{\Vert {\cal P}_{\delta,\gamma}^{-1} \Vert^2}
=
\underset{\bm{v} \neq 0}{\min} \frac{\Vert {\cal P}_{\delta,\gamma} \bm{v} \Vert^2}{\Vert \bm{v} \Vert^2}
\leq {\cal H}_{\delta,\gamma}(\xi) \leq
\underset{\bm{v} \neq 0}{\max} \frac{\Vert {\cal P}_{\delta,\gamma} \bm{v} \Vert^2}{\Vert \bm{v} \Vert^2} = \Vert {\cal P}_{\delta,\gamma} \Vert^2.
\end{align}

That is, any value of $1/{\cal H}_{\delta,\gamma}(\xi)$ serves as a lower bound on
$\Vert {\cal P}_{\delta,\gamma}^{-1} \Vert^2$, while any value of ${\cal
H}_{\delta,\gamma}(\xi)$ serves as a lower bound on $\Vert {\cal P}_{\delta,\gamma} \Vert^2$.
Therefore, the ratio of any two values of ${\cal H}_{\delta,\gamma}(\xi)$ provides a
lower bound on $\kappa^2({\cal P}_{\delta,\gamma})$.

We now show that bound \eqref{eq:kappa_gen_gamma*_bound} on $\kappa({\cal P}_{\delta,\gamma_*})$
is tight. Considering \eqref{eq:H_gen_def} at the judiciously chosen eigenvalues of
$\mathrm{i} \xi = \{ 0, \pm \mathrm{i} \sqrt{\delta \gamma_*}\}$, we have

\begin{align}
\label{eq:H1_gen}
{\cal H}_{\delta,\gamma}(0)
	&= \frac{\gamma_*^2}{\gamma^2},
\hspace{5ex}
{\cal H}_{\delta,\gamma}(\pm \sqrt{\delta \gamma_*}) =
	\frac{(2 \eta)^2 \gamma_*}{\delta( \gamma- \gamma_*)^2 + \gamma_* (\delta + \gamma)^2 }.
\end{align}

First observe from \eqref{eq:H_gen_bounds} and \eqref{eq:H1_gen} that $\Vert {\cal
P}_{\delta,\gamma_*} \Vert^2 \geq {\cal H}_{\delta,\gamma_*}(0) = 1$, and thus the upper bound
on $\Vert {\cal P}_{\delta,\gamma_*} \Vert$ from \eqref{eq:P_gen_gamma*_bounds} achieves
equality for a matrix $\cL$ having an eigenvalue of $\xi = 0$.
Secondly, observe from \eqref{eq:H_gen_bounds} and \eqref{eq:H1_gen} that $\Vert {\cal
P}_{\delta,\gamma_*}^{-1} \Vert^2 \geq 1/{\cal H}_{\delta,\gamma_*}(\pm \sqrt{\delta  \gamma_*})= [(\delta + \gamma_*)/(2\eta)]^2$, and thus the upper bound on $\Vert {\cal
P}_{\delta,\gamma_*}^{-1} \Vert$ from \eqref{eq:P_gen_gamma*_bounds} achieves equality for
a matrix $\cL$ having eigenvalues $\mathrm{i} \xi = \pm \mathrm{i} \sqrt{\delta \gamma_*}$.
Therefore, bound \eqref{eq:kappa_gen_gamma*_bound} on $\kappa({\cal P}_{\delta,\gamma_*})$
achieves equality for any matrix $\cL$ having eigenvalues $\{0, \pm
\mathrm{i}\sqrt{\delta \gamma_*}\}$.\footnote{By nature of the continuity of
eigenvalues and continuity of $\mathcal{H}_{\delta,\gamma}(\xi)$ at $\xi = 0$, there
also exist nonsingular matrices with condition number within
$\epsilon$ of \eqref{eq:kappa_gen_gamma*_bound} for any $\epsilon > 0$.}
\end{proof}

Last, having shown that \eqref{eq:kappa_gen_gamma*_bound} is tight, we now
show that  $\gamma = \gamma_*$ is optimal in terms of minimizing the maximum
condition number of all $\cL$ that satisfy \Cref{ass:fov}, by showing that
for $\gamma\neq\gamma_*$, $\exists$ matrices $\cL$ for which
$\kappa({\cal P}_{\delta,\gamma}) > (\delta+\gamma_*)/2\eta$.
\begin{lemma}[Optimal $\gamma_*$]
$\gamma_* =\argmin_{\gamma\in(0,\infty)} \max_{\cL} \kappa({\cal P}_{\delta, \gamma})$
\end{lemma}
\begin{proof}
Once again, consider the matrix $\cL$ from above
with eigenvalues $\{0, \pm \mathrm{i}\sqrt{\delta \gamma_*} \}$, such that
$\kappa({\cal P}_{\delta,\gamma_*})
=
(\delta + \gamma_*)/(2\eta)$. Now, observe from this, \eqref{eq:H_gen_bounds},
and \eqref{eq:H1_gen}, that for $0 < \gamma < \gamma_*$,

\begin{align}
\label{eq:kappa_gen_lwr_bnd1}
\kappa^2({\cal P}_{\delta,\gamma_*})
<
\frac{\gamma_*[\delta(\gamma - \gamma_*)^2 + \gamma_*(\delta + \gamma)^2]}{(2 \eta \gamma)^2}
=
\frac{{\cal H}_{\delta,\gamma}(0)}{{\cal H}_{\delta,\gamma}(\pm \sqrt{\delta \gamma_*})}
\leq
\kappa^2({\cal P}_{\delta,\gamma}),
\end{align}
by noting that the first inequality in \eqref{eq:kappa_gen_lwr_bnd1} is equivalent to
$\gamma_*(\gamma_* - \gamma)^2 + (\gamma_* - \gamma) [2 \gamma_*  \gamma + \delta(\gamma_* + \gamma)] > 0$, which is clearly true when $0 < \gamma < \gamma_*$.
Now suppose that $\cL$ has eigenvalues $\mathrm{i} \xi \to \pm \mathrm{i} \infty$,
which, when substituted into \eqref{eq:H_gen_def}, yields
$\lim_{\xi \to \pm \infty} {\cal H}_{\delta,\gamma}(\xi) = 1$.
Combining with \eqref{eq:H_gen_bounds} and \eqref{eq:H1_gen}, we have
for $\gamma_* < \gamma < \infty$,

\begin{align}
\label{eq:kappa_gen_lwr_bnd2}
\kappa^2({\cal P}_{\delta,\gamma_*})
<
\frac{\delta(\gamma - \gamma_*)^2 + \gamma_*(\delta + \gamma)^2}{(2 \eta )^2 \gamma_*}
=
\frac{{\cal H}_{\delta,\gamma}(\pm \infty)}{{\cal H}_{\delta,\gamma}(\pm \sqrt{\delta \gamma_*})}
\leq
\kappa^2({\cal P}_{\delta,\gamma}),
\end{align}
by noting that the first inequality in \eqref{eq:kappa_gen_lwr_bnd2} is equivalent to
$\delta (\gamma - \gamma_*)^2 + \gamma_*(\gamma - \gamma_*) (2 \delta + \gamma_* + \gamma) > 0$, which is clearly satisfied for $\gamma > \gamma_*$.

By construction in \eqref{eq:kappa_gen_lwr_bnd1} and \eqref{eq:kappa_gen_lwr_bnd2},
we have shown that for all $\gamma \in (0, \infty) \setminus \gamma_*$, there exist
matrices $\cL$ such that $\kappa({\cal P}_{\delta,\gamma}) > \kappa({\cal
P}_{\delta,\gamma_*}) = (\delta + \gamma_*)/(2\eta)$. It therefore holds for general $\cL$
satisfying \Cref{ass:fov} that a tight upper bound on $\kappa({\cal P}_{\delta,\gamma})$
for $\gamma \in (0, \infty) \setminus \gamma_*$ must be larger than the tight upper
bound of $\kappa({\cal P}_{\delta,\gamma_*}) \leq (\delta + \gamma_*)/(2\eta)$. Hence $\gamma=\gamma_*$ is the
minimizer over $\gamma \in (0, \infty)$ of a tight upper bound on $\kappa({\cal
P}_{\delta,\gamma})$.
\end{proof}
\end{proof}

\begin{proof}[Proof of \Cref{cor:cond}]
From \Cref{th:cond}, a tight upper bound on the condition number of ${\cal
P}_{\delta,\gamma}$ \eqref{eq:P_gen} {over all $\cL$} is minimized with respect to $\gamma$ when
$\gamma = \gamma_*$ \eqref{eq:gamma*_gen}, with its minimum value given by
\eqref{eq:kappa_gen_gamma*_bound}. To minimize bound
\eqref{eq:kappa_gen_gamma*_bound} with respect to $\delta$, we differentiate it
and observe for $\delta > 0$ that there is only one critical point at $\delta =
\sqrt{\eta^2+ \beta^2}$. Since this function is increasing as $\delta \to 0^+$
and $\delta \to \infty$, this critical point must be a local minimum. Therefore,
the tight upper bound \eqref{eq:kappa_gen_gamma*_bound} is minimized when
$\delta = \sqrt{\eta^2+ \beta^2}$. Substituting $\delta =  \sqrt{\eta^2+
\beta^2}$ into \eqref{eq:gamma*_gen} yields $\gamma_* = \sqrt{\eta^2 +
\beta^2}$. Finally, substituting $\delta = \gamma_* = \sqrt{\eta^2 + \beta^2}$
into \eqref{eq:kappa_gen_gamma*_bound} and noting that  ${\cal
P}_{\gamma,\gamma}$ \eqref{eq:P_gen} is equivalent to ${\cal P}_{\gamma}$
\eqref{eq:P_gamma} yields bound \eqref{eq:kappa_gamma*}.
\end{proof}

\section*{Acknowledgments}

This work was performed under the auspices of the U.S.\ Department of Energy by Lawrence Livermore National Laboratory under Contract DE-AC52-07NA27344 (LLNL-JRNL-817946).
Los Alamos National Laboratory report number LA-UR-20-30412.
This document was prepared as an account of work sponsored by an agency of the United States government.
Neither the United States government nor Lawrence Livermore National Security, LLC, nor any of their employees makes any warranty, expressed or implied, or assumes any legal liability or responsibility for the accuracy, completeness, or usefulness of any information, apparatus, product, or process disclosed, or represents that its use would not infringe privately owned rights.
Reference herein to any specific commercial product, process, or service by trade name, trademark, manufacturer, or otherwise does not necessarily constitute or imply its endorsement, recommendation, or favoring by the United States government or Lawrence Livermore National Security, LLC.
The views and opinions of authors expressed herein do not necessarily state or reflect those of the United States government or Lawrence Livermore National Security, LLC, and shall not be used for advertising or product endorsement purposes.

\bibliographystyle{siamplain}
\bibliography{refs2.bib}

\end{document}